\documentclass[leqno,12pt]{article} 
\usepackage{amsmath, amssymb}

\usepackage{amsthm} 
\def\N{{{\mathbb N}}}
\def\A{{{\mathbb A}}}

\def\Z{{{\mathbb Z}}}
\def\Q{{{\mathbb Q}}}
\def\R{{{\mathbb R}}}
\def\C{{{\mathbb C}}}
\def\P{{{\mathbb P}}}
\def\G{{{\mathbb G}}}

\def\O{{{\mathcal O}}}
\def\1C{{{\mathcal C}}}

\theoremstyle{plain} 
\newtheorem{theorem}{\indent\sc Theorem}[section]
\newtheorem{lemma}[theorem]{\indent\sc Lemma}
\newtheorem{corollary}[theorem]{\indent\sc Corollary}
\newtheorem{proposition}[theorem]{\indent\sc Proposition}

\theoremstyle{definition} 

\newtheorem{remark}[theorem]{\indent\sc Remark}
\newtheorem{example}[theorem]{\indent\sc Example}

\makeatletter
\def\address#1#2{\begingroup
\noindent\parbox[t]{7.8cm}{%
\small{\scshape\ignorespaces#1}\par\vskip1ex
\noindent\small{\itshape E-mail address}%
\/: #2\par\vskip4ex}\hfill%
\endgroup}%
\makeatother
\author{Pietro Corvaja,  Umberto Zannier}
\title{\uppercase{Rational and integral values of rational functions at rational points}}
\date{\today}
\begin{document}

\maketitle

\centerline{With an appendix by David Masser}

\begin{abstract}
 
 The basic issue of this paper concerns  {\it sets of values of rational functions at  rational points of  an  algebraic  variety}, namely images $f(X(k))$, where $X$ is an  algebraic variety and  $f:X\to \P_1$ is a rational map defined over the number field $k$ (and where we consider only the points in the domain of $f$).  
 
 We shall note that this problem reflects a multitude of different aspects, sometimes embracing previously studied questions and sometimes leading to new issues. 
 
For instance, we shall prove that {\it  for $X$ an abelian variety, the map between rational points is never surjective}. This is reminiscent of the Hilbert Property, but here the fibers may have arbitrary dimension. One of our examples concerns the classical Hilbert Property:  we produce {\it a simply connected affine surface whose set of integral points (over $\Z$)  is  Zariski-dense  but thin}, disproving a plausible expectation.

The general question turns out to be difficult already in the case $X=\P_2$, where   new open diophantine issues appear. In this context, we shall discuss a few cases concerning rational functions of low degree.

We shall also discuss heights and integrality issues; in this context, a role will be played by `gcd estimates'. 
  In the first Appendix, written by D. Masser, an effective estimate of some relevant gcd is provided.

 In the last part of the paper, we treat the issue of common values of rational or regular functions on varieties. We show for instance that two regular functions defined on a quasi projective surface can share all their (infinitely many) prime values without sharing  their whole value sets.
 
Many  problems we consider, although of very different origin, ultimately boil down to questions about integral points on cubic surfaces, or on del Pezzo surfaces of higher degree .
\end{abstract}

\section{Introduction} \label{S.Int} 

In this paper we will be basically concerned with the  {\it sets of values of rational functions at  rational points of  an  algebraic  variety}, namely images $f(X(k))$, where $X$ is an  algebraic variety and  $f:X\to \P_1$ is a rational map defined over $k$.  

Here $k$ is a number field, and we may think of it as being $\Q$; for many results this will make no essential difference, though  some examples will be treated only over $\Q$ or over $\Z$, which may be special in some respects. This difference will reflect in the organization of the paper.

\smallskip

{\tt Definition}. Throughout the paper, when the rational map $f:X\to \P_1$ is not everywhere regular, by $f(X(k))$ we mean the image of the set of rational points where $f$ is well-defined.

\smallskip

 A typical issue will be to understand whether the induced map $X(k)\to \P_1(k)$  between sets of $k$-rational points might be surjective, or ``almost'' surjective.
\smallskip

$\bullet$ When the variety $X$ is a curve, these questions enter the realm of the so-called Hilbert Irreducibility theory; although this theory will emerge later, here we shall be primarily concerned with the case $\dim X\geq 2$: hence we are concerned with the question of {\it rational points on fibrations}.  This issue had been considered in our book \cite{CZ-libro}, see Theorem 3.47. 
\footnote{More recently, G. Bresciani in \cite{B} proved some conditional result, depending on a strong uniform version of  Lang's conjecture on rational points on varieties of general type, in the case when the generic fibre of the projection $X\to \P_1$ are varieties of general type.}
When the total space $X$ is an abelian variety, we shall prove in particular that

\smallskip

 {\it A rational map $f:X\to \P_1$ is never surjective between rational points}. 
 
 \smallskip
 
 Note that there are situations where a  rational map  $f: X\to Y$, with $X$ possibly of greater dimension then $Y$, is surjective on rational points, even in absense of a rational section defined over the given number field of definition  ({\it see} e.g. Example \ref{ex.new} below).
 \smallskip

 \medskip

Here is a brief list of other topics touched in the paper:
 
 \smallskip

$\bullet$ Among our constructions, we shall produce fibrations in curves over the line (so that the total space will be a surface) such that infinitely many fibers admit exactly one rational point (avoiding trivial situations e.g. when there are base points). 

\smallskip

$\bullet$ We shall also be interested in values at {\it integral} points or in {\it integral values} at rational or integral points, of rational functions on $\P_n$ . 
Already for $n=2$ subtle problems arise. Let us see at once  a few  simple examples of this issue; see Section  \S \ref{S.rationalmaps} below for a more complete analysis and other examples.

- Consider the rational function in two variables 
$
f(x,y)=(x^2-1)/y^2,
$
 which can be viewed as a rational map $\A^2\to\A^1$. It turns  out that its integral values at integral points 
 are precisely the numbers $0,-1$ and all the positive integers which are not perfect squares.  

- For  the  similar rational function 
$
g(x,y)=(x^2+1)/y^2,
$ 
we shall observe that already the description is more recondite. 

- Consider now the function 
$ 
(x^2-1)/(y^2-1). 
$
 The set of its integral values at integral points includes  the set of positive non-square integers, but a precise description is somewhat surprising: see Example \ref{EX.0} below in Section \S \ref{S.rationalmaps}. 


$\bullet$ In Section \S \ref{S.rationalmaps} we shall consider other more elaborate examples. Let us notice at once that the last example with quadratic polynomials  is somewhat related  to a famous Olympiad problem (IMO 1988, Problem 6; see \cite{DJMP}), which (somewhat on the contrary) asked to  prove that all the integral values of the rational function
$$
\frac{x^2+y^2}{xy+1}
$$
at positive integer points {\it are} perfect squares. In Section \ref{S.simply-connected-surface} we shall show how the solution to this Olympiad problem leads to a {\it counter-example around the Hilbert property for certain simply connected surfaces}. 

\smallskip

Other divisibility issues, albeit of a more sophisticated nature, arose recently in the classification of cluster algebras and the correlated enumerations of friezes; see for instance the recent paper by  {\sc R. Zhang} \cite{Zh}, who considered in particular divisibility problems boiling down to integral points on cubic surfaces (special cases of which go back to {\sc Jacobsthal} and {\sc Mordell}). 

\smallskip


\smallskip

$\bullet$  In the last part of the paper, we shall study also {\it common values} of rational functions: namely, given {\it two} rational functions $f,g$ on a same variety, we shall investigate cases when  these functions share infinitely many values at rational (or integral) points, or even admit the same value set on such points. 
 
 We shall see that some natural problems in arithmetic or geometry can be phrased in terms of intersections (or comparison) of value  sets of functions at rational points; in turn, these questions lead to the study of distributions of rational or integral points on certain varieties, which had been extensively studied in classical Diophantine geometry. An example is represented by the investigation of rational and integral points on the Cayley cubic surface, which will be tackled while studying common values of homogeneous functions of degree one on the plane. Other examples arise in work of {\sc Zieve} and {\sc Carney-Horstch-Zieve}, however still in progress; they study the problem {\it whether given any number field $k$ there is a constant $c(k)$ such that for every nonconstant rational function $f$ on $\P_1$   such that every point outside a finite set $Z_f\subset\P_1(k)$ has at most $c(k)$ rational pre-images.}. And {\sc Zieve} raised  analogous issues over varieties other than $\P_1$.

\medskip

Let us introduce now some  formal setting.

\medskip

{\tt Value sets}. Let us fix a number field $k$. Given a quasi-projective algebraic variety $X$ over $k$ and a rational  map  $X\rightarrow \P_1$ we say that the set $f(X(k))$ is a {\it value set}. 

More generally, one can consider value sets in an arbitrary algebraic variety $Y$, provided by a rational map $X\to Y$, but the case of $\P_1$ will be basic here.

\medskip

As to $X$,  we do not  generally assume it to be irreducible or smooth. 

For our purposes, we may replace $X$ with the Zariski-closure of its set of rational points, so we may suppose from the outset that $X(k)$ is Zariski-dense.  In many cases, $X$ will be a semi-abelian variety.

\medskip

We note at once that every finite set is a value  set (a proof of this simple fact will be given later, since we do not know any precise reference) and that finite unions of value sets as well as finite intersections of value sets are value sets. 

\medskip

Strictly connected with the notion of value set is the important one of {\it thin set} (put forward by {\sc Serre}): in dimension $1$, a thin set in $\P_1(k)$ is a finite union of value sets $f(X(k))$ where we assume that $X$ has dimension $1$ and the rational map $f$ admits no rational section (in other words, the restriction of $f$ to each irreducible component of $X$ has degree $\geq 2$).  A similar definition holds over any irreducible variety $Y/k$ in place of $\P_1$, where we take rational maps $f:X\to Y$  from a variety $X$ of the same dimension as $Y$. 

 {\it Hilbert Irreducibility Theorem}  asserts that for $Y=\P_n$  the complement of thin sets is always Zariski-dense.

\smallskip

As we mentioned, many classical problems in Number Theory can be reduced to questions about value sets; recently, some  applications of the theory of value sets  to linear groups has been developped in \cite{CDRRZ}, \cite{CDRRZ2}; this concerns the case when $X$ is a linear torus, i.e. an affine semi-abelian variety.

\smallskip

\smallskip

As we already said, we shall look also at the analogous question for integral points.  

Values sets in $\A^1(\Z)$, defined as images of sets of  integral points on algebraic varieties, are called `Diophantine sets' in mathematical logic, and their study is at the base of the solution of Hilbert X problem about decidability of the existence of solutions for a general polynomial Diophantine equation. See e.g. the seminal paper by J. Robinson \cite{Rob}.

\medskip

\subsection{ Our main results}  Let $k$ be a number field, $X$ an algebraic variety and $f\in k(X)$ a rational function, all defined over $k$. 

 As alluded above, a prototype issue  for us concerns  for instance  {\it how large the  image set $f(X(k))$ is inside $\P_1(k)$}. In particular, we ask:  {\it when is this image the whole of (or a cofinite subset in) $\P_1(k)$?} By the Hilbert Irreducibility Theorem (which will be discussed in more  detail in subsection \ref{SS.Hilbert}) this cannot happen when $X$ is a curve, unless the map $f$ admits a rational section.
 
 A first result of this paper asserts that   this is never achieved also when $X$ is a semi-abelian variety of any dimension:
 
\begin{theorem}\label{T.Main}   For $i=1,\ldots , m$, let $X_i/k$ be a semiabelian variety with a rational map  $f_i:X_i\to \P_1$,    defined over the number field $k$. Let further $\Gamma_i$ be a finitely generated subgroup of $X_i(k)$. Then the complement $\P_1(k)-\bigcup_{i=1}^mf_i(\Gamma_i)$  contains infinitely many rational integers. In particular,   the complement $\P_1(k)-\bigcup_{i=1}^mf_i(\Gamma_i)$   is Zariski-dense.

\end{theorem}

Note that since $X_i$ is an abelian variety, we can take for $\Gamma_i$ the whole set $X_i(k)$ of its rational points,  by the Mordell-Weil theorem; in the case of semi-abelian varieties, we could rephrase the statement by considering the sets of $S$-{\it integral} points, with respect to any ring of $S$-integers of $k$.

\medskip

An explicit infinite set of  integral points in the complement set $\P_1(k)-\bigcup_{i=1}^mf_i(\Gamma_i)$ can always be described, as the following refinement to Theorem \ref{T.Main} shows:

\begin{theorem}[Addendum to Theorem \ref{T.Main}]\label{Addendum}
Under the assumptions of Theorem 1.1, for all but finitely many primes $p$   the set $\{p,p^2,p^3,\ldots\}$ of the powers of $p$ has finite intersection with the set $f(X(k))$. 

Also, letting $\mathcal{F}=\{0,1,2,3,5,8,\ldots\}$ be the set of Fibonacci numbers, either $\mathcal{F}\cap f(X(k))$ is  finite or has finite complement in $\mathcal{F}$.
The same holds when $\mathcal{F}$ is replaced by the set of values of any binary recurrent sequence of exponential growth. 
\end{theorem}

\begin{example}
Let $E$ be an elliptic curve over $\Q$, say of Weierstrass equation
\begin{equation}\label{E.ell-curve}
y^2= x^3+ax+b,  
\end{equation}
where the polynomial $x^3+ax+b\in\Q[x]$ has   no multiple factors. Let   $F\in\Q(x,z_1,z_2)$ be a rational function. Then:
\smallskip

{\it  There are infinitely many $t_0\in\Q$ such that the equation 
$F(x,2^m,3^n)=t_0$ has no solutions such that $m,n\in\Z$, and  $x=x(P)$  for some point $P\in E(\Q)$. }

\smallskip

This may be obtained from Theorem \ref{T.Main}  by taking $X=E\times\G_{\rm m}^2$ and  by taking $f((x,y),z_1,z_2)=F(x,z_1,z_2)$. 

We note that changing $E$ with $\P_1$ would generally lead to problems which to our knowledge are difficult and widely open. For instance, the case $f(x,z_1,z_2)= x^2-az_1-bz_2$ leads to the issue of understanding {\it whether}  for `many'   integers  $t_0$ there are  no integers $m,n$ such that $t_0+a2^m+b3^n$ is a square. Probabilistic considerations would suggest the existence of infinitely many such $t_0$. Possibly, congruence considerations can help here. However, already replacing $2^m,3^n$ by more general $S$-units  
we do not see a possible path for proving this (the approach of \cite{B} would succeed, but it assumes strong diophantine conjectures). 
 \end{example}

\smallskip

\begin{example}\label{Ex:somme}
Another application of Theorem \ref{T.Main} entails the following: \smallskip

{\it Given an elliptic curve as in equation \eqref{E.ell-curve}, there exist infinitely many integers $t\in\Q$ not of the form $t=x(P)+x(Q)$, for any  choice of points $P,Q\in E(\Q)$.
}
\smallskip

Under the celebrated Vojta's conjectures, to be discussed later, a much stronger result could be proved: {\it   only finitely many integers can be written in the form $x(P)+x(Q)$ for rational points $P,Q\in E(\Q)$}.
\end{example}

\medskip

A curious application of Theorem \ref{T.Main} provides an example of an infinite family of algebraic curves, each admitting exactly one smooth rational point.

\begin{proposition}\label{Prop.1-solo}  There exists a pencil of smooth curves of genus $>1$, defined over $\Q$, on an algebraic surface $X$, without base point, such that infinitely many members defined over $\Q$ have exactly one  rational point. Also, one can construct $X$ so that its set of rational points is Zariski-dense.
\end{proposition}

\smallskip

Note that easier examples can be constructed without the hypothesis of non-existence of base points: namely, consider a surface with only finitely many rational points and a pencil of curves having as base point one of these rational points (and no other). Then almost every such curve will have exactly one rational point. 

Another trivial case occurs  in presence of a (rational) base point for the pencil: after suitably blowing up the surface over such a point, one can obtain a base-free pencil. However,  the unique rational point on the generic element of such a pencil would arise from a rational section.

\medskip

{\tt Considerations on heights}. The conclusion provided by Theorem \ref{T.Main} is in a sense intuitive, if not for the following reason: for a finitely generated multiplicative group, the height in $\Gamma$ grows rapidly. If we knew that the height in $f(\Gamma)$ also grows rapidly, then $f(\Gamma)$ would fail to contain many points in $\P_1(k)$.

 However, putting this approach on a rigorous ground misses a piece, that is, a proof of the fast growth of the height in $f(\Gamma)$. This amounts to a lower bound  for $h(f(x))$ in terms of $h(x)$. Such a lower bound is not difficult to prove (for $x$ in an open dense set) for maps of  finite degree, but is otherwise subtle and mostly unknown; one can obtain a good information assuming well-known conjectures of Vojta (see \cite{BG}). 
We shall develop some arguments and details about this problem in \S\ref{S.H} and in   Appendix 1 (written by D. Masser).

For the moment we merely note the following result, which is a corollary of the proof of  Theorem \ref{T.Main} and quantifies it in terms of heights.

\begin{theorem} \label{T.H1} In the situation of Theorem \ref{T.Main}, let   us  set 
\begin{equation}\label{E.lower-bound}
N(T):=\# \{x\in \P_1(k)-\bigcup_{i=1}^mf_i(\Gamma_i), h(x)\le T\}.
\end{equation}

Then, for every $c>0$ we have  $N(T)\gg_cT^c$. 
\end{theorem} 
As we shall see in section \ref{S.H},  this is a rather weak result compared to what would be obtained through Vojta's conjectures (under which one can often obtain a sharp bound for the distribution of points of given height in the complementary set inside $\P_1(k)$).

We note that a result of this type, in the particular case of tori, was at the base of the tools exploited in \cite{CDRRZ} to study bounded generation on linear groups.

We shall discuss height issues in  section \ref{S.H}.

\subsection{Integrality and divisibility issues}

{\tt Integral points}. 
On  some occasions,  we shall allow $X$ to be quasi-projective, and, for a number field $k$,  we shall consider only  the $S$-integral points of $X$ in place of the whole $X(k)$.  

The $S$-integral points are  defined here as follows.  Suppose that $X$ is embedded as an open subset of a projective variety $X^*\subset \P_n$, and  set $D=X^*-X$, the (support of the)  divisor at infinity. Suppose also that $X,X^*$ have good reduction outside a finite set $S$ of places of $k$. Then we define the $S$-integral points (or simply `integral' when $S$ is given) as those points in $X (k)\cap \P_n(k)$ such that their reduction falls outside the reduction of $D$. So, for $X$ projective we  have an empty $D$ and we find back the whole of $X(k)$.  

Recall that for semiabelian varieties   the $S$-integral points form  automatically a finitely generated group. 

\medskip

{\tt $S$-integer values of a rational function}. We can also look at the $S$-integer values of $f$. This corresponds to restricting $f$ to $S$-integer points of another variety, as follows. First, let $\tilde X$ be a blow-up of $X$ such that there is a morphism  $\tilde f:\tilde X\to\P_1$ which coincides with $f$ on the  part of $\tilde X$ which corresponds to the non-blown up part of $X$. Then we may consider $\tilde Y=\tilde X-\tilde f^*(\infty)$; namely, we remove from $\tilde X$ the pull-back of the point at infinity in $\P_1$. Then, essentially by definition, the  integral points of $\tilde Y$ are those integral points of $\tilde X$ where moreover $\tilde f$ assumes integral values. We shall see examples below.

\smallskip

 As shown in \cite{CZadv}, questions about integral points on rational surfaces  can sometimes be viewed as questions on divisibility between values of polynomials. Even the $S$-unit equation theorem of Siegel and Mahler can be viewed in this sense: the $S$-unit solutions $(u,v)$ to the equation $u+v=1$ correspond bijectively to the solutions $u$ in $S$-integers to the divisibility problem: $u|1$ and $(1-u)|1$. This was the starting point to our investigations in \cite{CZadv}.
 
 Even the `elementary' problems stated in the introduction, e.g. the problem of determining the integral values of the function $(x^2-1)/y^2$ can be reduced to that of finding integral points on an affine cubic surface, namely the surface of equation $zy^2=x^2-1$.

The mentioned famous Olympiad problem, concerning integral values of the quotient $(x^2+y^2)/(xy+1)$, $x,y$ being natural numbers, leads to the study of the   equation $z(xy+1)=x^2+y^2$, which represents another affine smooth cubic surface. 

\smallskip

$\bullet$ The general problem of describing  the  set of integral values at integral points of a rational function  $f(x,y)=P(x,y)/Q(x,y)$, $P(x,y), Q(x,y)\in\Z[x,y]$, i.e. the set $f(\Z^2)\cap \Z$, seems very difficult already for $P,Q$ of small degree $\geq 2$. In general, even understanding   whether this set is the whole of $\Z$ seems to fall outside the available techniques. We already mentioned cases related to the theory of Pell's equation, where the property of {\it not} being a square is equivalent to being a value. As a sample of our methods, we shall prove in  \S \ref{S.rationalmaps}  that
\medskip

\begin{theorem}\label{T.divisibility}
Let $P(x,y)\in\Z[x,y]$ be a polynomial with integral coefficients, of total degree $2$, which is $\geq 0$ in $\R^2$. Let $Q(x,y)\in \Z[x,y]$ be a polynomial with integral coefficients whose zero set is a hyperbola with irrational asymptotes. Consider the ratio
$$
q:=\frac{P(x,y)}{Q(x,y)}
$$
If there exists a point $(x_0,y_0)\in \Z^2$ such that $q\in \Z$ and $(x_0,y_0)$ is not the center of symmetry of the   hyperbola  of equation $Q(x,y)=0$ , then the set of  pairs $(x,y) $ such that $q\in\Z$ is not thin (in particular it is Zariski-dense in the plane). Also, the corresponding set of integers $q$ is not thin.
\end{theorem}  

\smallskip

As a consequence, we obtain  for instance that the set of  pairs $(x,y)\in\Z^2$ such that the ratio  
$$
q=\frac{x^2+y^2}{x^2-2y^2-1}
$$
is in $\Z$    is  not thin (indeed the point $(1,1)$ is a solution to the divisibility problem, distinct from the centre of symmetry of the zero set of the denominator). 
  
  \smallskip

 As mentioned before, recently the study of cluster algebras and the related problem of classifying friezes  lead to considering divisibility problems and systems of divisibility conditions, where only {\it positive} integral solutions are allowed (some of these problems go back to {\sc Mordell}): see e.g. the paper \cite{Zh} by {\sc R. Zhang}. We owe to him the reference \cite{Moh} 
  by {\sc Mohani}, who classified for instance  the solutions to the system $x|y^3+1,\, y|x^3+1$.  Further forthcoming work on some of these issues is due to {\sc Kollar-Li} \cite{KL}.

 \medskip

In some of our results it is essential that the integral points in question have  coordinates in $\Z$; in other cases, on the contrary, one can generalize the finiteness or degeneracy results that we obtain to $S$-integral points   on an arbitrary number field. Of course the latter situation is more geometric, but the former one contains many appealing issues special of $\Z$, which could be still significant. 

For instance, in the last mentioned example by {\sc Mohani}, it is clear that the pairs of $S$-integers $(x,y)$ satisfying the divisibility conditions $x|(y^3+1),\, y|(x^3+1)$ are Zariski-dense if the ring of $S$-integers contains infinitely many units. The same is true for the initial divisibility problems linked to Pell's equation.

In geometrical terms, the affine surface whose integral points correspond to the solutions of these divisibility conditions, which is a suitable blow-up of the affine plane, has the property of containing a Zariski-dense set of integral points over suitably large   rings of $S$-integers. For similar divisibility problems with {\it three} divisibility conditions  the form $P_i(x,y)|Q_i(x,y)$, $i=1,2,3$, where $P_i(x,y),Q_i(x,y)\in \Q[x,y]$ are polynomial satisfying suitable conditions, the authors proved in \cite{CZadv} that the set of pairs of $S$-integers satisfying the said divisibility problem in the ring of $S$-integers is always degenerate or finite.

\subsection{Hilbert Property}\label{SS.Hilbert} 

Let us consider the general situation of a rational map $f: X\to \P_1$, defined over a number field $k$.   When $X$ is an irreducible curve and $\deg f>1$, the set  $f(X(k))$ of   values is called {\it thin} (over $k$), after a (more general) definition put forward  by {\sc Serre} (see \cite{S}, \cite{S3} or \cite{CZ-libro}); actually, thin sets (of $\P_1$, over $k$) are by definition finite unions of sets of such a shape.\footnote{By definition thin sets - over $k$ -  in a variety  include also those contained in a proper subvariety; but it is easy to see that these last ones are of the former type as well, see \cite{CZ-libro}.}  
This is relevant in the theory of  Hilbert Irreducibility Theorem, a simple version of which  may be  formulated by saying that  {\it $\P_1(k)$ is not thin}.  In rough terms, we may further rephrase this by saying that {\it  not all rational points of $\P_1$ may be lifted to a rational point  on at least one among finitely many {\tt  finite} covers of $\P_1$ given in advance, each of degree $>1$}.   

And there are several ways to quantify this type of assertion  (for which we refer to \cite{S} and \cite{BG}). For instance, we note that for a map $f:X\to \P_1$, where $X$ is a curve, we have, for the (logarithmic)  heights, $h(f(p))\ge (\deg f) h(p)+O(\sqrt{h(p)})$ ($p\in X(\overline\Q)$), where $h(p)=h_{P_0}(p)$ refers to a height with respect to a given point $P_0\in X(\overline\Q)$. Hence {\it  the height  of  a value} of a point  grows rapidly enough in terms of {\it the height of the point}. 
For instance if $X$ has positive genus we have that the number of points $p\in X(\Q)$ with $h(p)\le T$ is bounded by a fixed power of $T$ (and of course it is $O(1)$ for genus $>1$), so  in particular this implies that  many rational target-values are necessarily left out (since the number of points of $\P_1(k)$ having height $\leq T$ is $\gg \exp(T)$).  Similarly if we consider only target  points  in a finitely generated subgroup of $\G_{\rm m}$. (Note that for genus $0$ we need however further arguments. We shall comment below on what can be said in higher dimensions by this type of considerations.) 

\medskip

{\tt Higher dimensions}. Of course, when $\dim X>1$, the general fiber of a map $f$ as above has positive dimension, hence in a sense  it is much more likely to find rational points in a given fiber than in the case when $X$ is a curve. And indeed it is very easy to produce several examples, more or less trivial,  when $f(X(k))$ covers $\P_1(k)$ or most of it. For instance  this certainly holds when the map $f$ has a rational section defined over $k$, i.e. a rational map  $g:\P_1\to X$ such that  $f\circ g$ is the identity of $\P_1$.  

In some cases, the existence of such a section is a necessary condition for the image $f(X(k))$ to cover (a `large enough' subset of) $\P_1(k)$.  For $k=\Q$, a  nice instance is provided by a theorem of {\sc Davenport, Lewis, Schinzel}, when $X$ is a pencil of conics, given by an equation $X:A(t)x^2+B(t)y^2=z^2$ in $\P_2\times \P_1$, $A,B\in\Q(t)$, $f((x:y:z)\times t)=t$ being the second projection.  If $k=\Q$ they proved that a section exists provided only that $f(X(\Q))$ contains  a set of integer points (in $\A^1(\Z)$) meeting every arithmetical progression (in other words, the conic $X_n= X\cap \{t=n\}$  has rational points, for all integers  $n$ running  through such a set). The issue here leads into the theory of specialization of Brauer groups, studied in depth by many authors;   we refer e.g. to the paper \cite{ZEns} of the second author  for a discussion (with some proofs, also over arbitrary number fields) and for several references.  

This kind of result may be seen as a local-global principle over a 
global field (analogue to the case of ternary forms of the Hasse-Minkowski principle, since we are assuming that a certain ternary form over $\Q(t)$ has rational points for many `reductions' $t\to t_0\in\Q$). It may fail on replacing $\P_1$ by a curve with infinitely many rational points e.g. already for a pencil of conics over  an elliptic curve; for instance, let $E:y^2=x^3+5x$ and  take $X\subset \P_2\times E$ defined by $u^2+v^2=xw^2$, $f$ given again by the second projection. See \cite{CZ-libro}, Thm. 3.46, for a proof that   $f(X(\Q))$ covers $E(\Q)$ (which is infinite) whereas there are no rational sections defined over $\Q$.\footnote{By a well-known theorem of Tsen, for any  conic-bundle  over a curve there are always sections defined over  a suitable finite extension of the ground field - the present case being obvious on factoring the left side over $\Q(i)$; here we have an instance when the choice of the ground number field is relevant.}

  
\smallskip

\begin{example}\label{ex.new}
Another example of surjectivity, based on different principles, can be constructed as follows: Let us start from a plane affine curve of equation $f(x,y)=0$ admitting infinitely many integral points $(x,y)\in\Z^2$, where $f(x,y)\in\Z[x,y]$ is a polynomial such that $f(0,0)=1$.
Let $q(u,v)\in\Z[u,v]$ be the norm-form of a  quadratic ring of class number $1$ (e.g. $q(u,v)=u^2+v^2$). Let $S\subset \A^4$ be the surface
$$
S=\{(x,y,u,v)\in \A^4\, :\, f(x,y)=0,\, xy=q(u,v)\}.
$$
Finally, let $X\subset\A^6$ be the threefold defined as
$$
X=\{ (x,y,a,b,c,d)\in\A^4\, :\, f(x,y)=0,\, x=q(a,b),\, y=q(c,d)\}.
$$
Since the quadratic form $q(u,v)$ is muultiplicative, the threefold $X$ is endowed with a natural projection $X\to S$, obtained by expressing the product $q(a,b)\cdot q(c,d)$ in the form $q(u,v)$, where $u,v$ are bilinear forms in terms of $(a,b),(c,d)$. 

The fact that the coordinates $x,y$ of every integral point $(x,y,a,b)\in\ S(\Z)$  are coprime (as follows from $f(0,0)=1$) implies that the product $xy$ is representable by the binary quadratic form $q$  if and only if $x$ and $y$ are so.

\smallskip

Note that, both in this last example and in the mentioned Example 3.46 from \cite{CZ-libro}, the fibration  does admit a rational section,  although defined only over a suitable finite extension of $\Q$.
\end{example}

The Hilbert property in higher dimensions is defined similarly: given a variety $Y$ over a number field $k$ we say that a subset $Z\subset Y(k)$ is thin over $k$ if there exists a possibly reducible algebraic variety $X$ of the same dimension and a rational map $f: X\to Y$ without rational sections such that $Z\subset f(X(k))$. Whenever $Y(k)$ is not thin, we say that the variety $Y$ satisfies 
the Hilbert Property over $k$. All rational varieties do satisfy the Hilbert Property, essentially by Hilbert Irreducibility Theorem.

\medskip

A result presented in this work, turning around value sets of integral points,  concerns an issue which we raised in \cite{CZhilb}. {\it Given a smooth simply connected algebraic variety $Y$ over a number field $k$, with $Y(k)$ Zariski-dense, is it true that $Y(k)$ is not thin?} (We say also that the Hilbert Property holds for the set $Y(k)$ over $k$.) Or, at least is it true that after   a suitable extension of a the number field $k'/k$, the set of $k'$-rational poinits $Y(k')$ is not thin (over $k'$)?  

The analogous problem can be  stated for integral points. As we proved in \cite{CZhilb}, the condition that $Y$ is simply connected cannot be removed.

Note that in this situation the variety $Y$ plays the role of the line: we are implicitly considering maps $X\to Y$, where $X$ has the same dimension of $Y$ and the map is dominant without rational sections.

We shall show that

\begin{theorem}\label{T.dense-and-thin}
There exists a simply-connected smooth affine cubic surface $S\subset \A^3$ defined over $\Q$ such that the set $S(\Z)$ is Zariski-dense and thin.
\end{theorem}

This shows that, at least for the case of integral points, one needs to allow finite extensions of the field of definition, or at least of the involved ring of $S$-integers, in order to obtain the Hilbert Property from the Zariski-density.

A theorem of {\sc Coccia} \cite{Coccia} asserts that for such surfaces the Hilbert Property does hold  after enlarging the field of definition, or replacing the ring $\Z$ by a suitable finitely generated subring of $\Q$.

\medskip

Let us show a closer link between our Theorem \ref{T.Main} and    the Hilbert Property. 
 Thinking of the case of curves and what we have mentioned regarding thin sets, Theorem \ref{T.Main} may be viewed as a higher-dimensional extension of Hilbert Irreducibility Theorem.
 
 Indeed, consider a finite-degree rational map $\pi:A\to \A^g\subset \P_g$, where $A$ is an abelian variety of dimension $g>1$, say all defined over $\Q$. 
 We can write $\pi(x)=(f_1(x),...,f_g(x))$ for rational maps $f_i:A\to\P_1$. 
 
  Now, the Hilbert Property for $\P_g$ ensures the existence of (many) rational points $u=(u_1,\ldots ,u_g)\in\A^g(\Q)$ such that $\pi^{-1}(u)$ does not contain points in $A(\Q)$.  The present result  in particular ensures the much stronger fact that this holds for all rational $u$ in a denumerable (and complex-dense) union of affine hyperplanes. In fact, we obtain that there exist points $u_1\in\Q$ such that already $V:=f_1^{-1}(u_1)$ (which generally has codimension $1$ in $A$) does not contain points in $A(\Q)$ where $f_1$ is defined;  hence,  for all $u$ in the hyperplane of $\A_1^g$ defined by $x_1=u_1$ (which is the closure of $ \pi(V)$), $\pi^{-1}(u)$ does not contain points in $A(\Q)$.  

\medskip

\subsection{ A more detailed overview of the paper} 

Let us go back to our initial problem of investigating the value sets $f(X(k))$, where $f:X \rightarrow \P_1$ is a rational map.

{\tt Warning: about the domain of $f$}. We remark   once again that,    for simplicity of notation, we shall adopt a slight  abuse of language  and   use  the notation `$f(U)$' even for sets $U\subset X(k)$ not necessarily contained in the domain of definition of $f$;  this means that  we tacitly {\it do not consider the points of $U$ where $f$ is not defined}; i.e., we mean $f(U\cap\hbox{domain of $f$ in $X$})$.  Correspondingly, we shall usually exclude these points when we take fibers of $f$, hence by `$f^{-1}(V)$' we think of $f$ as a function on its domain in $X$.  If we want to refer to the whole fiber we can use its closure in $X$, denoted $\overline{f^{-1}(V)}$. 

\medskip

Note that the set of points where $f$ is {\it not} defined is closed in $X$, and  has  codimension $\ge 2$ as soon as $X$ is smooth. Of course we may blow-up $X$ along this set and go to a blown-up variety $\tilde X$ with a regular map $\tilde f$ to $\P_1$, which equals our $f$ on the previous domain. We shall adopt this procedure for  some arguments. 

However the regularised map might be surjective on rational points for trivial reasons, namely taking the image of the exceptional divisor.   The already mentioned example of the rational map $\A^2\to \A^1$ given by the formula $(x^2-1)/(y^2-1)$ is paradigmatic: the regularized map is trivially surjective, since the equation $(x^2-1)=m(y^2-1)$ always admits the solution $(1,1)$.
\medskip

{\tt A main question}: As above, a question that seems natural to ask  is {\it whether (almost) all   $k$-rational points of $\P_1$ are in the said image},  i.e. are  attained by $f$ at some rational point of $X(k)$ where $f$ is defined. 

Our main purpose in this paper will be to show in some natural cases that {\it many  points in $\P_1(k)$ are NOT attained by $f$ on $X(k)$}.

 This is a bit vague. But for the moment let us still  briefly discuss this in some cases.

\medskip

{\tt Local obstructions}. An easy way to provide a proof of the said expectation, i.e., a negative  answer to the question, is through local obstructions. Namely, it suffices to find a prime $p$, a place $v$ of $k$ above $p$,  and a point $\rho \in\P_1(k)$, such that $\rho\not\in f(X(k_v))$, equivalently the fiber $f^{-1}(\rho)$, i.e. the variety $f=\rho$, has no point over $k_v$; in practice, this amounts to the non-solvability of   certain congruences.

This method may well work sometimes, but there are many natural cases where it is not suitable  to work. And moreover one may   reformulate the question on restricting  the possible values to those having no local obstruction.  

\medskip

  In the main part of the paper we focus on the case when $X$ is an abelian (or semi-abelian) variety (or more generally when $f$ factors through a map from an abelian variety); in the abelian  case of course $X(k)$ is finitely generated. More generally, we may take $X$ to be a semiabelian variety and ask about the image $f(\Gamma)$ for a finitely generated subgroup $\Gamma\subset X$.  In this context we have the formidable  tool furnished by the Theorem of Faltings (former Mordell-Lang conjecture), asserting that: 

 {\it For a subvariety $Z$ of an  abelian variety $X$  and a finitely generated subgroup $\Gamma\subset X(k)$,  
 the set $Z\cap\Gamma$ is contained in a finite union of translates of abelian subvarieties}. 
 
 An equivalent and maybe more direct statement is that:  {\it The Zariski closure of any  set $\Sigma\subset \Gamma$ is a finite union of translates of abelian subvarieties}.  
 
 Vojta (to whom is due a  basic input in Faltings' method for these results) later proved an extension to semiabelian varieties, whereas the toric case was known earlier due to the work of W. Schmidt, whose Subspace Theorem was used efficiently by Evertse, Schlickewei-van der Poorten, and M. Laurent - among others - in this context.

\medskip

{\tt Weak Approximation}.  From these deep  results we shall   deduce (see Theorem \ref{T.Main2}) that  for a semiabelian variety $X/k$, the set $f(\Gamma)$ is {\it sparse}, in the sense that its complement is `large': by this here we mean that such complement not merely is not-thin, but   satisfies the (stronger) so-called 
 {\it weak-approximation} property:  for a set $A\subset \P_1(k)$, this means that $A$ is dense in $\prod_{v\not\in M_k}\P_1(k_v)$; equivalently, for every finite set $S$ of places of $k$,  any collection $\alpha_v, v\in S$, and every $\epsilon >0$, there is $x\in A$ such that $|x-\alpha_v|_v<\epsilon$ for each $v\in S$. 
 
 Recall also that the weak-approximation for a set $A\subset V(k)$ (where $V/k$ is any algebraic variety) indeed  implies that $A$ is not thin (see \cite{S}, Ch. 3, where it is indeed  shown that for being not thin  even the {\it weak-weak-approximation} suffices, namely it suffices that  the density holds for every $S$ disjoint from a prescribed finite set $S_0$ of places). 

\medskip
 
 Actually, in our context  we can even prove a kind of  {\it strong-approximation} property in $\A^1$ with respect to the ring of integers $\O_k$ of $k$. Let $\O_k$ (resp. $\O_{v,k}$) denote the ring of integers of $k$ (resp. $k_v$). Then, {\it  in this paper} this property, for a set $A\subset \A^1(\O_k)=\O_k$, means that $A$ is dense in $\prod_{v\in M^0_{k}}\O_{v,k}$, where $M^0_{k}$ is the set of finite places of $k$.\footnote{We take this meaning for avoiding the truble of introducing further definitions; the result may be strengthened as in the usual meaning of strong approximation.}
 
And we can prove this sparseness even on removing from $\P_1(k)$ not merely one but an arbitrary  finite number of such  images (as in the theory of thin sets). We have:

\begin{theorem} \label{T.1} For $i=1,\ldots, m$, let $X_i/k$ be a semiabelian variety with a rational map  $f_i:X_i\to \P_1$,    defined over the number field $k$. Let further $\Gamma_i$ be a finitely generated subgroup of $X_i$. Then the complement $\P_1(k)-\bigcup_{i=1}^mf_i(\Gamma_i)$ has the weak-approximation property over $k$.

In particular, such complement is not thin over $k$.

\end{theorem}

 A special case of this result, containing the basic principles of the present proof, already appeared (in weaker shape)  in the book \cite{CZ-libro}: see Thm. 3.47. 
 (The fact that the complement is not thin however was not mentioned, but could be proved directly on expanding just by a few words the argument of the proof.)  
 
 \medskip

 \begin{remark} We point out that the proof leads to a somewhat more general result, as may be easily seen by inspection, and indeed will be remarked through the proof. Namely, omitting the index $i$ for clarity, $f(\Gamma)$ can be replaced by a set obtained as follows. Take a closed hypersurface $H\subset X\times \P_1$ and remove from it the sets $\{x\}\times \P_1$ entirely contained in it, obtaining $H^*$. Then, in place of $f(\Gamma)$ we may consider the set $\pi_2(H^*(k))$.  
 
 The above statement is obtained by taking as $H$ the closure of the graph of $f$. We have preferred to formulate only the present result for simplicity and brevity. 
 
 \end{remark}

\subsection{ Curves with precisely one smooth rational point (Proof of Proposition \ref{Prop.1-solo})}   
 By the same  simple  idea  for the proof of Theorem \ref{T.1}, other applications which seem not to be obvious at first sight. For instance, we now show  the existence of pencils of curves, {\it without base point}, such that infinitely many of them have exactly one smooth rational point (Proposition \ref{Prop.1-solo}).



Note that it is easy to construct infinitely many such curves starting with an abelian variety having just one rational point, and taking a pencil of curves passing through that point. However, all these curves will have a same rational point, while we aim at producing examples where the (unique) rational point does depend on the curve (hence the hypothesis on the absense of base points for the pencil).

\smallskip

To sketch a proof of  Proposition \ref{Prop.1-solo}, let $A/\Q$ as above be an  abelian surface with trivial endomorphism ring and Zariski-dense set of rational points. Let also  $f:A\to\P_1$ be a nonconstant rational map defined over $\Q$, constructed as in the opening remarks of this example. So we also suppose that the points of indeterminacy of $f$ are not rational, and that $f$ does not factor through another map $\P_1\to\P_1$ of degree $>1$. 

 By taking $f$ sufficiently general we may also achieve that it 
does not satisfy any nontrivial  invariance $f(q\pm x)=f(x)$. In fact, in the first place the above defined divisor  $D$ may be invariant only for finitely many translations. Therefore, it first suffices to exclude that $f$ is invariant by these translations; in turn, this follows from the fact that, since $pD$ is very ample for large $p$, the functions in the linear system separate points, thus the ones invariant by a given translation form a proper sub-vector space of $L(pD)$. Similarly for excluding that $f$ obeys the other invariance (with the minus sign). 

Let us now consider the rational points on the subvariety $X\subset A\times A$ defined by taking the closure of  $f(x)=f(y)$ in $A$. By   Faltings' Theorem   they all lie on finitely many translates of abelian subvarieties of $A\times A$ contained in $X$, so let us analyse these translates.

By our hypothesis on $A$, the only nontrivial abelian subvarieties of $A\times A$ are defined by $A_{m,n}=\{(x,y)\in A^2: mx=ny\}$, where $m,n$ are coprime integers. Suppose that  $X$ contains a translate $t+A_{m,n}$. This happens when $mn=0$: indeed, e.g. $X$ contains $s\times A$ and $A\times s$ for every point $s\in A$ where $f$ is not defined (and only for such $s$). Suppose now $mn\neq 0$. Then $f$ must be  clearly invariant by translation by $m$-torsion or $n$-torsion points, hence by $mn$-torsion points. 
By taking $f$ sufficiently general as above,   we may suppose that $m=n=1$ and that $X$ contains only the diagonal among the said translates.


Take now as our curves $C_t$ the closures of the fibers $f^{-1}(t)$. If $x,y\in C_t(\Q)$ then $t\in\Q$ and either $f$ is not defined at $x$, or $y$, which we are excluding, or $f(x)=f(y)$. So  $(x,y)\in X(\Q)$. But this yields $x=y$, as required. Note that since $A(\Q)$ is Zariski-dense, we obtain in fact infinitely many such curves $C_t$ which (if again $f$ is chosen in a `general' way) are smooth outside the points of indeterminacy of $f$, and having one rational point (and hence exactly one). 

\smallskip


\section{Value  sets and thin sets} 

We recall some definitions and simple properties of value sets.

We defined value sets in an algebraic variety $Y$ over a number field $k$ as images of sets of $k$-rational points of another irreducible variety   $X$ under a morphism.

The next result shows that every finite set is a value set:

 \begin{proposition}\label{P.finite}
 Let $Y$ be an irreducible quasi-projective algebraic variety defined over a number field $k$. Let $F\subset Y(k)$ be a finite set. There exists an irreducible algebraic variety $X$ of the same dimension as $Y$ and a dominant rational map $\pi: X\to Y$, both defined over $k$, such that  $F=\pi(X(k))$.
\end{proposition}  

Finite sets in an algebraic variety are just algebraic subvarieties of dimension zero. The above results extends to the following one:

\begin{proposition}\label{P.finite2}
Let $Y$ be an absolutely irreducible quasi-projective algebraic variety defined over a number field $k$ and let $Z\subset Y$ be a closed algebraic sub-variety defined over $k$. Then there exists an absolutely irreducible algebraic variety $X$ of the same dimension as $Y$ and a dominant rational map $\pi: X\to Y$, both defined over $k$, such that  $Z(k)=\pi(X(k))$.
\end{proposition}

\begin{proof}
Let us start from the case $Y$ is an affine normal variety. Let $Z$ be defined by a system of equations $f_1=\ldots=f_k=0$, where for each $i=1,\ldots,k$, $f_i$ is a regular function on $Y$. Suppose also that the zero divisor of each function $f_i$ is reduced. 

Choose one plane algebraic curve  over $k$ whose only rational point is the origin and such that this point is smooth; let $P(u,v)=0$ be its equation, for a polynomial $P(u,v)\in k[u,v]$. Consider the compositum of the algebraic covers of $Y$ given by the equations $P(u,f_i(x))=0$, and note that in view of our assumption on the zero set of the $f_i$, each such cover is absolutely irreducible. In  terms of function fields, we add to the ring $k[Y]$ all the algebraic functions $u_i$ satisfying the equations $P(u_i,f_i)=0$, and then we take its integral closure. With little effort, we can obtain a compositum which is still absolutely irreducible; if this were not the case, one can just modify the curve $P(u,v)=0$, taking possibly a different curve for each function $f_i$. 

We finally obtain a variety $X$ endowed with a map $X\to Y$. Its rational points correspond to the values of $y\in Y$ such that each of the equations $P(u,f_i(y))=0$ admits a rational solution $u=u_i\in k$. Due to our assumption, this can happen only if $f_i(y)=0$ for all $i$, i.e. if $y\in Z$, as wanted.

The theorem is then proved for affine varieties.

Let us treat now the projective case (whenever $Y$ is neither projective nor affine, the problem can be easly reduced to the projective case).

One can always find a hypersurface $W\subset Y$ defined over $k$, not containing the given subvariety $Z$,  such that $Y \setminus  W$ is affine and $W$ does not contain any $k$-rational point: indeed, just embed $Y\hookrightarrow \P_N$ in some projective space and then find a hypersurface $W'$ with this property in $\P_N$ (for instance using local conditions to exclude the presence of $k$-rational points); the intersection $W'\cap Y$ will be the required hypersurface $W$. 

Now, consider the affine variety $Y\setminus W$. The cover $X\to (Y\setminus W)$ constructed above can be completed to a morphism $\bar{X}\to Y$, where $\bar{X}$ is a suitable completion of $X$. The rational points on $X$ will all lie above $Z$.

\end{proof}

 \medskip
  
{\tt Thin sets}. Let $X$ be an absolutely irreducible algebraic variety over a number field $k$. Whenever a set of rational points $Z\subset X(k)$ is contained in a value set $Z'\subset X(k)$ which  can be obtained as an image of a finite union of morphisms $\pi_i:Y_i\to X$, $i=1,\ldots,n$, where for each $i=1,\ldots,n$,  $Y_i$ is an irreducible variety of the same dimension as $X$ and $\pi_i:Y_i\to X$ is a dominant map of degree $>1$, we say that the value set $Z$ is {\it thin}.

By what we proved in Proposition \ref{P.finite2}, if $Z$ is not Zariski-dense then it is thin. Also, finite union of thin sets are thin.

Sometimes, one defines thin sets as those contained in an image $\pi(Y(k))$, for a single morphism $\pi:Y\to X$ where $Y$ is no more  supposed to be irreducible. In that case the condition on $\pi$ is that it admits no rational section.




\section{Heights and Vojta's Conjecture}\label{S.H} 

In this section we exploit in a little more detail the viewpoint sketched in the Introduction, i.e., that one could prove the sparseness of values by proving first that the height of the values grows fast. Note that lower bounds for the heights proved crucial in several Diophantine problems; recent general results have been provided by {\sc Dimitrov-Gao-Habegger} \cite{DGH} and by {\sc Yuan-Zhang} \cite{YuZh} which admitted applications e.g. in the context of uniform versions of Mordell's conjecture.

In general, it is very difficult to achieve strong lower bounds, and we shall merely give an example relying on a very deep conjecture of {\sc Vojta}.

Then we shall also give a few results holding unconditionally.  

\medskip

Let $h$ denote the Weil height on $\P_1$, hence the height with respect to the point at infinity for the standard coordinates.

We start with an example relying on Vojta's conjectures, and for definiteness let us take an abelian surface $A$ over a number field $k$, assumed to be with trivial endomorhism ring (the generic case) and with a rational map $f:A\to\P_1$. 

Of course the fibers of $f$ will be curves, and hence we  have not bounded heights for their algebraic points, whence no lower bound for $h(f(x))$ in terms of $h(x)$ can   hold. Nevertheless one would hope that such a lower bound would still hold restricting the points $x\in A(\overline\Q)$, for instance taking merely those over a given number field $k$. Let us see what we can obtain from the usual functorial properties. 

We may assume that $f$ is regular outside a finite subset $\Phi$ of $A(\overline\Q)$. Let $\pi: \tilde A\to A$ be a suitable blow-up variety of $A$ above $\Phi$ such that there is a morphism $\tilde f:\tilde A\to \P_1$  with $\tilde f= f\circ \pi$ outside $\pi^{-1}(\Phi)$. Let also $E$  denote the exceptional divisor (which might be reducible).

 By the basic  functorial property of heights, we have, for $x\in A-\Phi$ and $\tilde x=\pi^{-1}(x)\in\tilde A$, 
\begin{equation*}
h(f(x))=h(\tilde f(\tilde x))=h_{{\tilde f}^*(\infty)}(\tilde x).
\end{equation*}
Hence we are let to investigate the divisor ${\tilde f}^*(\infty)$ of $\tilde A$.  By general theory, this is linearly equivalent to the sum {\it  (strict transform of fibre of $f$)$+E$}.  In turn, this is equivalent to $\pi^*(\hbox{fibre of $f$}) +$ linear combination   of components of $E$.  

A  divisor  class of the shape $\pi^*(\hbox{fibre of $f$}) $ equals the sum of an ample divisor class $B$ and another linear combination of components of $E$. 

A conjecture of Vojta (see \cite{BG}, Conj. 14.3.2, p. 483)  implies, for the points over a given number field $k$, 
\begin{equation*}
h_{K_{\tilde A}}\le \epsilon h_B+O(1),
\end{equation*}
outside a proper subvariety of $\tilde A$, where $K_{\tilde A}$ is the canonical class. But this is equal also to a sum of components of $E$ (since $K_A=0$).  We deduce
\begin{equation*}
 h_{{\tilde f}^*(\infty)}(\tilde x)\ge (1-\epsilon)h_B(\tilde x)+O(1),
\end{equation*}
again outside a proper subvariety of $\tilde A$, and finally we obtain the following

\begin{proposition} Let   $f:A\to \P_1$ be rational map from an abelian surface $A$, defined over a number field. Under Vojta's Conjectures, for every $\epsilon>0$ there exists a Zariski-dense set $U\subset A$ such that for all $x\in U(k)$  
\begin{equation}\label{E.H} 
 h(f(x))\ge (1-\epsilon)\hat h(x)+O(1),
\end{equation}
where $\hat h$ is a canonical height on $A$. 
\end{proposition}
\smallskip

Now, if $\Gamma$ is finitely generated subgroup of $A(\overline\Q)$, say of rank $r$, we have
\begin{equation*}
\#\{x\in \Gamma: \hat h(x)\le T\}\ll T^{r\over 2},
\end{equation*}
whence, by \eqref{E.H}, 
\begin{equation*}
\#\{y\in  f(\Gamma): h(y)\ll_c T\}\ll T^{cr},
\end{equation*}
for any $c>1/2$.  Taking into account that the number of $y\in\P_1(k)$ with $h(y)\le T$, even merely those which are integers, is $\gg  \exp(c' T)$ (for a $c'>0$), we see that most values in $k$ are not attained by $f$ on $\Gamma$.  This greatly improves on the estimate of Theorem \ref{T.H1}, but of course is subject to the widely open difficult conjecture of Vojta.

\medskip
{\tt Common denominators and $\gcd$ estimates}. A nice consequence of this instance of Vojta's conjecture concerns common denominators of rational points on elliptic curves.
Given an elliptic curve defined by a cubic equation with integral coefficients, say
$$
y^2 = x^3 +ax+b
$$
for integers $a,b\in\Z$, the coordinates of  every rational point $(x,y)$ can be written as fractions of the form $x=u/d^2, y=v/d^3$, for a  positive integer $d$ such that $\gcd(u,v,d)=1$. We call $d$ the denominator of the rational point $P=(x,y)$ and denote it by $d(P)$.

As a consequence of Vojta's conjecture, one could prove that there do not exist infinitely many pairs of points $P,Q$ on an elliptic curve with $d(P)=d(Q)$ and $P\neq \pm Q$ (see the discussion in \cite{CZ-libro}, pag 85). 
Hence 
\smallskip
 
{\it Let $E$ be an elliptic curve over $\Q$. Under Vojta's Conjectures,  only finitely many sums of the form $x_1+x_2$, where $x_1,x_2$ are $x$-coordinates of rational points on $E$, are integers.}

 \smallskip

See Example \ref{Ex:somme} and \cite{McKinnon} for a proof of the relevant case of Vojta's conjecture in the case $E(\Q)$ has rank one. 
\medskip

Let us state now a few results of a similar strength, also cases of Vojta's Conjecture,  which may be proved unconditionally.  They constitute the analogue of the estimates \ref{E.H}, where abelian surfaces are replaced by $\G_{\rm m}^2$, and were first proved by the  authors; the  result included the following

\begin{theorem} [see \cite{CZgcd}, Thm. 1]\label{T.lower-bound-heights-1} Let $\Gamma$ be a finitely generated subgroup of $\G_{\rm m}^2(\overline\Q)$ and let $f,g\in \overline\Q[u,v]$ be nonconstant coprime polynomials not both vanishing at $(0,0)$.  For all $\epsilon >0$ there exists a finite union $Z$ of translates of proper subtori of $\G_{\rm m}^2$ such that, for all $p\in\Gamma -Z$ we have
$$
h\left({f(p)\over g(p)}\right) >h(f(p):g(p):1)-\epsilon h(p).
$$
\end{theorem}

The proof worked first by bounding a $\gcd$ (under an appropriate notion) of $f(p),g(p)$.   We note that the affine height  $h(f(p),g(p),1)$ can be bounded below essentially by the maximal height of the monomials that appear in $f,g$: this is a consequence of the Schmidt Subspace Theorem, after work of Schlickewei-van der Poorten and Evertse (see \cite{CZ-libro} for more). The present result instead bounds the cancellation {\it outside} the finite set $S$ of places such that $\Gamma$ consists of $S$-units.  

A previous result regarded the case when $f,g$ were resp. $u-1,v-1$. (See again \cite{CZ-libro} for a brief history of these results, which started with an upper bound for $\gcd(a^n-1,b^m-1)$ obtained also with  Bugeaud in \cite{BCZ}.)  The link of Vojta's Conjecture was noticed by  Silverman in \cite{Silverman}. 

The methods did not extend neither to the case of abelian surfaces in place of $\G_{\rm m}^2$, nor to higher dimensional tori.  While the first issue remains open, the second one was resolved by Levin, who proved in \cite{L} an analogue for two polynomials in an arbitrary number of variables.

\smallskip

The proof of Theorem \ref{T.lower-bound-heights-1} relies on the Subspace Theorem in Diophantine approximation, hence the result  is not effective.

Weaker estimates can be made effective by the use of Baker's theorem of linear forms in logarithms. As a sample, in  Appendix 1 by D. Masser it will be proved that
$$
h\left(\frac{3^m-1}{2^n-1}\right)\geq 10^{-32}\max\{m,n\} -1
$$
for all non-negative integers $m,n$ with $n\geq 1$.

\smallskip

As mentioned,  in \cite{BCZ} it was proved an ineffective result of the form $\gcd(a^n-1, b^n-1)\ll_\epsilon e^{\epsilon n}$, for multiplicatively independent positive integers $a,b$, where the (ineffective) implied constant depends on $a,b,\epsilon$. In particular the quotient $(a^n-1)/(b^n-1)$ cannot be integral for large $n$. 

Masser's estimate above is too weak to imply the non-integrality  of the quotient; however, replacing $3$ and $2$ by larger numbers $a>b$ whose ratio $a/b$ is sufficiently close to $1$, one can effectively bound the largest $n$ such that the quotient $(a^n-1)/(b^n-1)$ is integral.
This is done in Appendix 2, where we prove for instance that

\smallskip

\begin{theorem}\label{T.appendix}
For every pair of positive integers $a,b$ with $1<b<a<b(1+2^{-55})$ and for every integer $n\geq 1$, the ratio $(a^n-1)/(b^n-1)$ is not an integer. 
\end{theorem} 
\smallskip

This  is reminiscent of lower bounds for the fractional part of powers of $a^n/b^n$, where the lower bound $\|a^n/b^n\|\gg_\epsilon e^{-\epsilon n}$ holds for every $\epsilon$, as an application of Roth-Ridout's theorem,  and effective bounds can be provided whenever $a/b$ is sufficiently close to $1$.

\section{Proofs of main results}

Let us prove Theorem \ref{T.1}, in the following form:

\begin{theorem}\label{T.Main2}
Let $X$ be a a semi-abelian variety defined over a number field $k$; let $f:X\to \P_1$ be a non-constant map. Let $\Gamma\subset X(k)$ be a finitely generated subgroup. The set of points in $X(k)$ which are not on the image $f(\Gamma)$ satisfy the weak-approximation property. In particular, they form an infinite set.  
\end{theorem}

Note that in  this theorem, unlike  Theorem \ref{T.Main}, we work with only one semi-abelian variety. Actually, the case of several semi-abelian varieties and several maps can be easily formally recovered: suppose indeed that $h$ semi-abelian varieties $X_1,\ldots,X_h$ are provided  with non-constant maps $f_i:X_i\to \P_i$ and finitely generated subgroups $\Gamma_1\subset X_1(h),\ldots, \Gamma_k\subset X_h(k)$. Select non-zero values $a_1,\ldots,a_k$ attained by the $f_i$ at some point in $\Gamma_i$, for $i=1,\ldots,h$. Let $X_{h+1}$ be the torus $\G_m$ and $\Gamma_{h+1}\subset X_{h+1}$ be the multiplicative group generated by $a_1,\ldots,a_h$ and $f_{h+1}:X_{h+1}\to \P_1$ the canonical inclusion of $\G_m$ into $\P_1$. 
Set $X=X_1\times\cdots \times X_{h+1}$ and let $f:X\to \P_1$ be defined as $f(x_1,\ldots,x_{h+1})=f_1(x_1)\cdots f_{h+1}(x_{h+1})$. 
Finally, set $\Gamma=\Gamma_1\times\cdots \Gamma_{h+1}\subset X(k)$. Then it is clear that the value set  $f(\Gamma)$ includes each value set $f_i(\Gamma_i)$, for $i=1,\ldots,h$.

\medskip

In the proof of Theorem \ref{T.Main2} we shall use the following simple lemma, which we state explicitly although  it is well known:

\begin{lemma}
There exist  infinitely many $\bar{\Q}$-isogeny classes of elliptic curves over a given number field. In each class, there exists an elliptic curve over $k$ of positive Mordell-Weil rank.
\end{lemma}
 
\begin{proof}
To show the infinitude of isogeny classes, probably the simplest way is observing that for two isogeneous elliptic curves each $j$-invariant is integral over the ring generated by the other. Indeed, the two $j$ invariants are linked by a modular equation $\Phi_N(x,y)=0$, where $\Phi_N(x;y)\in\Z[x,y]$ is monic in each variable.   


Let us now consider a single elliptic curve over $k$; we shall find a second one, still defined over $k$ and isomorphic to the given one over $\bar{k}$, having positive Mordell-Weil rank. We can argue as follows: consider a Weierstrass equation of the form
$y^2=p(x)$, for a degree-three polynomial $p(x)$ with distinct roots. Every specialization $x\mapsto x_o\in k$ gives rise to a pair of  quadratic algebraic points on the elliptic curve (with vanishing sum). Only finitely many of them have finite order, so we can pick one of them, say $(x_0,y_0)$, of infinite order. Observe that $y_0^2$ belongs to $k$ and consider the twist $y_0^2 y^2=p(x)$; this new elliptic curve admits the rational point $(x_0,1)$, which is of infinite order. 
\end{proof}
\smallskip

\begin{proof}[ Proof of Theorem \ref{T.Main2}] 

Let $X,k,f$ be as in the statement. Let us choose a finite set of non-archimedean places $\{\nu_1,\ldots,\nu_h\}$ of $k$, rational points $q_1,\ldots,q_k$ and positive integers  $r_1,\ldots,r_h$. We look for points $p \in \P_1(k)\setminus f(\Gamma)$ which  approximate   the points $p_1,\ldots,p_h$ with respect to the given valuations, i.e. satisfy
\begin{equation}\label{E.approximation}
p\equiv q_i \pmod {\nu_i^{r_i}} \qquad \mathrm{for} \ i=1,\ldots,h.
\end{equation} 
By the weak approximation property on the line (i.e. the Chinese remainder theorem), we can find a point $p_0\in \P_1(k)$ such that the above system of congruences is satisfied for $p=p_0$.

Then, construct an elliptic curve $E$ over $k$, of positive Mordell-Weil rank, and a morphism $\pi:E\to \P_1$, still defined over $k$, such that $\pi$ sends the neutral element of $E$ to $p_0$. We also require that $E$ admits no non-constant morphism to $X$. 
It is clear that such an elliptic curve can be constructed: by Poincar\'e's complete reducibility theorem   the elliptic curves admitting non-constant morphisms to $X$ lie in only finitely many isogeny classes and then by the previous lemma one can pick an elliptic curve over $k$ (even over $\Q$) outside this finite union of classes, having positive Mordell-Weil rank. 

Now, we claim that only finitely rational points of $\P_1$ can lie both in the value sets   $\pi(E(k))$ and $f(\Gamma)$: indeed, consider the fibre product 
$$
Y=X\times_{f,\pi} E=\{(x,u)\in X\times E\, :\, f(x)=\pi(u)\}\subset X\times E.
$$
It is a hypersurface of the semi-abelian variety $X\times E$ and is endowed with a natural map to $\P_1$, sending $(x,u)\mapsto f(x)=\pi(u)$. The points in $f(X(k))\cap \pi(E(k))$ are precisely the images of the rational points in $Y$. Now, by Faltings-Vojta's theorem on rational points on sub-varieties of semi-abelian varieties, the set $Y(k)$ is contained in the union of finitely many translates of algebraic subgroups of $X\times E$ contained in $Y$.

By our assumption on $E$ the only connected algebraic subgroups of the product $X\times E$ are of the form  $A\times E$ or $A\times\{0_E\}$,  where $A\subset X$ is an algebraic subgroup of $X$. Let us show that  $Y$ does not contain any translate of a  subgroup of the form $A\times E$. Indeed,   if for some $x_0 \in X\times E$ the translate $(u_0+A)\times E$  would be contained in $Y$, from the equation $f(u_0+a)=\pi(u)$ valid for all $a\in A$ and $u\in E$, we would obtain a contradiction (since $\pi$ is non-constant).

Then all rational points in $Y(k)$ are contained in sets of the form $X(k)\times F$, for a finite set $F\subset E(k)$. In other words, the projection $\pi: X\times E \to \P_1$ takes only finitely many values in the points of $Y(k)$. Equivalently, for all but finitely many points $u\in E(k)$, the rational point $\pi(u)$ does not lie in the value set $f(X(k))$, and {\it a fortiori} not in the image $f(\Gamma)$. 

To construct the rational point $p\in \P_1(k)\setminus f(\Gamma)$ approximating the given points $q_1,\ldots,q_h$ as in \eqref{E.approximation}, observe that $p=\pi(0_E)$ does solve the system \eqref{E.approximation}. Take now any non-torsion point $u_0\in E(k)$; its multiples $nu_0$, for sufficiently divisible $n\in Z$, are congruent to the neutral element $0_E$ modulo $\nu_i^{r_i}$ for $i=1,\ldots,h$. Then their images $\pi(nu_0)$ solve the congruence system \eqref{E.approximation}, and all but finitely many of them lie outside the set $f(\Gamma)$.

\end{proof}

{\it Proof of Theorem \ref{T.H1}}.  We shall now sketch a  {\it proof of Theorem \ref{T.H1}}. Recall that, in the notation of Theorem \ref{T.Main2}, we now want to prove that the set $\P_1(k)\setminus f(\Gamma)$ has a growth bounded from below by inequality \eqref{E.lower-bound}. 

The construction carried out so far provides a set of rational points outside $f(\Gamma)$ which grows, as a function of the height, as the set of rational points of an elliptic curve; if $r$ is the rank of the elliptic curve used in the above proof, the number of rational points so constructed of logarithmic height   $\leq T$ is $\gg r^{1/2}$. Replacing the elliptic curve by an abelian variety of arbitrary dimension leads to substantially the same proof. One must only take further care in excluding the presence of algebraic subgroups in the relevant product, which might provide infinite families of points in $\P_1(k)$ which lift to the sami-abelian variety. However, this is not difficult to achieve, either by chosing a simple abelian variety nor isogenous to any factor of the given  semi-abelian variety $A$ or chosing products (even power) $E^r$ of an elliptic curve $E$ as in the proof above. In this last case, although the product  $A\times E^r$ will contain many algebraic subgroups, for a generic choice of the map $\pi:E^r\to \P_1$, such subgroups will not be contained in the hypersurface of equation $f(x)=\pi(y)$ in $A\times E^r$.

\medskip

 {\it Proof of the Addendum (Theorem \ref{Addendum})}.  
In the proof of Theorem \ref{T.Main2}, it is shown that the complement of $f(\Gamma)$ in $\P_1(k)$ contains all but finitely many images of rational points of rational functions defined on an auxiliary elliptic curve.   These infinite sets of rational points, lying on  the complement of the value set $f(\Gamma)$, can indeed be chosen with some  flexibility: for instance, one can find such sets which are composed of rational integers.

The idea is to replace the auxiliary elliptic curve $E$ by a a suitable linear torus. Let us show the details, following the pattern (and the notation) of the proof of Theorem \ref{T.Main2}. 

Given a semi-abelian variety $X$ and a rational map $f:X\to \P_1$ as in the statement of Theorem\ref{T.Main2}, we can find a non-split linear torus $G$ defined over $\Q$, having infinitely many integral points (via the theory of Pell's equation) and admitting no non-constant  morphism  $G\to X$ {\it defined over the number field $k$}. This is obvious if $X$ is an abelian variety. Otherwise, $X$ admits a maximal connected linear subgroup $H$, which is isogenous over $k$ to the products of finitely many non-decomposable linear tori. 

It suffices to take for $G$ a $\Q$-torus that is not isomorphic over $k$ to any factor of $H$. Recall that $G(\Z)$ is infinite (basically, this reduces to   the theory of Pell's equation).

We now choose a non-constant morphism $\pi:G\to \A^1\subset \P_1$, given by a regular function on $G$ defined over the rational integers,  and define the hypersurface $Y\subset X\times G$ as the fiber product of $(f:X\to \P_1)$ and $(\pi:G\to \P_1)$. The rest of the proof is identical.

We then obtain that all but finitely many numbers of the form $\pi(u)$, for $u\in G(\Z)$, lie outside the image of $f(X(k))$ and all these numbers are rational integers.

When $G=\mathrm{SO}(q)$ is the orthogonal group of the quadratic form $q(x,y)=x^2+xy-y^2$ and  $\pi:G\to \A^1$ is the map sending $G\ni \left( \begin{matrix} a&b \\ b&a+b\end{matrix}\right)\mapsto b$, the image $\pi(G(\Z))$ contains the set of Fibonacci numbers.
Any other   binary linear recurrent sequence, provided it is not an arithmetic progression, can be treated the same way.

\medskip

Another variant of the proof, giving rise to the first conclusion of the Addendum, runs as follows: consider simply the multiplicative group $\G_m$ over $k$. Although $\G_m$ might be a factor of the maximal linear torus $H$ inside the semi-abelian variety  $X$, since the group $X(\O_S)$ is finitely generated, for every proper algebraic subgroup $K$ of $\G_m\times X$ defined over $k$ and all large primes $p$, there are no points in $K(\Q)$ of the form $(p,x)$ with $x\in X(\O_S)$.

This proves the first finiteness conclusion of the Addendum.

\section{Examples of integral values  over $\Z$ of  rational maps $f:\A^2\to\A^1$} \label{S.rationalmaps} 

 In this section we collect some examples of integrality of values, using rational maps $f:\P_2\to\P_1$ or $f:\A^2\to\A^1$ (this distinction  means that in the first case we consider {\it rational points} where a rational function takes integral values (with respect to $\infty\in\P_1$ for the standard coordinates), in the second case {\it  integral points} where such rational function is integral valued).
 
 Each of the examples contains some proofs, but we haven't stated them as `theorems' since they are different and special in many aspects. Because of the complication in the arguments, an exception occurs with  Example \ref{EX.3}, where we shall provide a proof of Theorem \ref{T.divisibility}.  We shall also restrict to the ring $\Z$, which for most  features in  these examples  behaves in a peculiar way compared to general number fields.

\medskip

In all the present examples we set $f={P(x,y)\over Q(x,y)}$, where $P,Q\in\Z[x,y]$ are coprime polynomials of degree $\le 2$, and $x,y$ are affine coordinates (so, if we work in $\P_2$, we have therein homogeneous coordinates $(x:y:1)$). (We may assume that the homogeneous part of $Q$ of degree $2$ is not definite, for otherwise the ratio is bounded on $\R^2$.) We see that we have integral values at points where $Q$ is a unit. This also   immediately justifies our comment above on the peculiarity of $\Z$: if we work in arbitrary  rings of $S$-integers, the equation $Q(x,y)=$unit  will more often have a Zariski-dense set of solutions for $\deg Q=2$. 

\smallskip

{\tt Pencils of affine conics}. Note further that, if we let $q$ be an integral value, the equation becomes 
$P(x,y)=qQ(x,y)$, so,  if for instance we seek integral  $x,y$, we have a pencil of  affine conics over $\A^1$, defined over $\Z$, and we ask for which integral values of the parameter we have an integral point. 

\smallskip

In general this appears to be a deep issue. As already mentioned, the case of {\it rational} solutions amounts to conics having rational points, and has been studied in the context of specialisation of Brauer groups, by a number of authors (see e.g. Serre's paper \cite{S2} and  the last author's  paper \cite{ZEns}, also for references).  

\smallskip

From another viewpoint, one can consider the solutions $(x,y)$ to the divisibility problem as being integral points on a suitable del Pezzo surface of degree $5$ with respect to an anti-canonical divisor: indeed, blowing up the four points of intersection of the two conics $P(x,y)=0$ and $Q(x,y)=0$, one obtains a complete del Pezzo surface of degree $5$ (its anti-canonical embedding sends the blown-up plane to a surface of degree $5$ in $\P_5$) The integrality of the quotient $P(x,y)/Q(x,y)$ amounts to the integrality of the corresponding point $(x,y)$ with respect to the line at infinity plus the strict transform of the conic $Q(x,y)=0$. This sum is a reducible divisor in the anti-canonical class (a hyperplane section with respect to the said embedding inot $\P_5$).

The {\it potential} density of the integral points for all del Pezzo surfaces with an anti-canonical divisor at infinity has been proved by Coccia in \cite{Coccia2}, after preliminary work of Hassett-Tschinkel. Actually in \cite{Coccia2} the potential Hilbert Property is also proved, in the case the surface is simply connected, which is the present case.

\smallskip

Yet another viewpoint consists in cosidering the affine cubic surface of equation $ P(x,y)=z\cdot Q(x,y)$, whose integral points are the solutions to our divisibility problem. Note that such a surface, although can be smooth in $\A^3$, has a singularity at infinity (with respect to the embedding $(x,y,z)\mapsto (x:y:z:1)$ it is the point $(0:0:1:0)$).

\smallskip

{\tt Sections}. Note also that by Tsen's theorem a pencil of conics has always a rational section defined over some finite extension of the ground field, so the problem for rational points is naturally  heavily  linked with  the special ground field. On the contrary, seldom there are integral sections for an affine pencil, no matter the ground field. (Examples come from the theory of the Pell-Abel equation $x^2-D(t)y^2=1$, see \cite{Zpell}.)  The study of regular sections from $\A^1$ is a topic of independent interest and it seems worthwhile to develop it together with an analysis  of the integral points much more extended than the present one.


\smallskip

\medskip

Going back to the above questions, we now present some examples treated in detail, as announced in the introduction.

\medskip

\begin{example}\label{EX.-2} {\bf Divisibility of $x^2-1$ by $y^2$}.
Set   $
f(x,y)={x^2-1\over y^2},
$
 which can be viewed as a rational map $\A^2\to\A^1$, for which we look  at  integral values at integral points. If $q\in \Z$ is such a value, then the equation $x^2-qy^2=1$ admits an integral solution $(x,y)\in \Z^2$  with $y\neq 0$ (note that the solutions   $(\pm 1,0)$ are not admissible). Then either $q\in\{0,-1\}$ or $q$ is a positive non-square number. Also, all the positive non-squares integers are attained values at integral points, since the corresponding Pell's equation admits non-trivial solutions. Hence the relevant set consists precisely of the numbers $0,-1$ and all the positive numbers which are not perfect squares.   
 
 The set of points $(x,y)\in\Z^2$ so obtained is not thin in $\A^2$, as is not difficult to check.
 
 {\tt Rational points}. The issue of  nonzero integral values at rational points leads to the conic $x^2=qy^2+z^2$, which admits the rational point $(1:0:1)$, so is rational over $\Q$ for each integer $q\neq 0$.
\end{example}

\begin{example}\label{EX.-1} {\bf Divisibility of $x^2+1$ by $y^2$}.
Let  now 
$
g(x,y)={x^2+1\over y^2}.
$ 
The integral values attained by $g$ at integral points with $y\neq 0$ are, apart from the number $1$,  those positive non-square integers $q$ such that the ring $\Z[\sqrt{q}]$ admits a unit of norm $-1$. Contrary to the previous case, their distribution is still in part mysterious, though it is well known that  their set includes primes $\equiv 1\pmod 4$ and excludes numbers divisible by any  prime $\equiv 3\pmod 4$.   Note the inclusion $g(\Z^2)\cap \Z \subset f(\Z^2)\cap \Z$, which is not explained by any factorization of the map $g$ by $f$.

Again, the set of corresponding points $(x,y)\in\Z^2$   is not thin in $\A^2$.

 {\tt Rational points}. We are led to consder the conic $x^2+z^2=qy^2$, which has rational points if and only if the integer $q$ is the sum of two squares in $\Z$.  
\end{example}

\begin{example}\label{EX.0} {\bf Divisibility of $x^2-1$ by $y^2-1$}. Consider now the rational function $
{x^2-1\over y^2-1}$,  excluding the values $ x,y= \pm 1
$ where it is not defined.

 The set of its integral values at integral points includes  the set of positive non-square integers: indeed, for every positive integer $q$, the equation $(x^2-1)/(y^2-1)=q$ is equivalent to the  equation $x^2-qy^2=1-q$, provided that the `trivial' solutions $(\pm 1,\pm 1)$ 
are  excluded. However the existence of the trivial solution $(1,1)$ provides an integral point on the hyperbola  defined by the previous equation, and this fact, in turn, entails the existence of infinitely many other solutions if $q>0$ is not a square, since the   hyperbola  has irrational asymptotes.\footnote{This is a special case of a general theorem of Gauss: see {\it Disquisitiones}, Art. 216.}  On the other hand, $q<0$ implies $y= 0$  and gives rise to the values  $q$   of the form $1-n^2$, for $n=1,2,\ldots$.
  As to the squares $q=n^2$, the equation  amounts to  $x^2-(y^2-1)n^2=1$, which may be viewed as a Pell equation with discriminant $y^2-1$ (and solution $(x,n)$). For varying integer $y>1$, all the integer solutions are then given by $x+n\sqrt{y^2-1}=(y+\sqrt{y^2-1})^t$, for all nonzero integer values of the exponent $t$. (See also the related Example \ref{EX.1} below.) So the relevant $n$ are given by the Chebyshev polynomials of the second kind: 
  $$
  n=U_t(y):={(y+\sqrt{y^2-1})^t-(y-\sqrt{y^2-1})^t\over 2\sqrt{y^2-1}}, \qquad 1<y\in\N, \quad t=1,2,\ldots.
  $$
The set of the integers $n$  so obtained includes the even integers (set $t=2$) and  it is not difficult to show that  the subset of odd numbers contains $O(\sqrt B)$ integers in $[-B,B]$ but  is not thin in $\A^1$.

It is also not difficult to show that the set of integer points $(x,y)$ giving an integral value is not thin in $\A^2$.

\begin{remark}
One simple way of obtaining an integral value for $(x^2-1)/(y^2-1)$ is to choose $y\neq\pm 1$ and then $x\equiv \pm y \pmod{y^2-1}$. In geometric term, the surface $S$ of equation $x^2-1=z(y^2-1)$ is covered by a pencil of curves parametrized by $\A^1$, namely those obtained by sending
$$
\A^1 \ni y \mapsto (x,y,z)=(y+t(y^2-1),y,2ty+t^2(y^2-1)+1) \in S,
$$
where $t$ is a parameter.
In other words, the map $\A^2\ni(t,y)\mapsto (x,y,z)\in S$ defined by the above formula provides a birational morphism  from the plane to the cubic surface $S$.  Another parametrization comes from te case $x\equiv -y \pmod{(y^2-1)}$.

Note that such a parametrization {\it does not} produce all integral points; indeed, suppose $y^2-1$ factors into a product of coprime numbers as $y^2-1=mn$, $\gcd(m,n)=1$. Then there are  solutions with $x\equiv y$ $\pmod m$ and $x\equiv -y\pmod n$. 

In geometric terms, the failure of the morphism to cover all integral points is explained from the fact that its inverse is    rational but non-regular.  
\end{remark}

 {\tt Rational points}. We find now the conic $x^2+(q-1)z^2=qy^2$ in $\P_2$.  Again, this is rational over $\Q$ in view of the rational point $(1:1:1)$.
 \end{example}



\begin{example}\label{EX.1} {\bf Divisibility of $x^2+y^2-a$ by $xy$}.  Let $P(x,y)=x^2+y^2-a$, $Q(x,y)=xy$. If  $q\in\Z$ is an integral value of $f(x,y)=P(x,y)/Q(x,y)$ at an integer point $(x,y)\in\Z^2$,  we have $x^2+y^2-qxy=a$, i.e. $(2x-qy)^2-(q^2-4)y^2=4a$.  We have a kind of almost-Pell equation of discriminant $q^2-4$. If $|q|=0,1$ we have only finitely many integer solutions.  If $q=\pm 2$, there are solutions if and only if $a$ is a square, and the integer  solutions lie on the  two  lines $x-qy=\pm\sqrt a$. If $|q|>2$, the theory of Pell equation indeed yields that if there is an integral solution there are infinitely many. Also, a solution with least $|y|$ satisfies $y^2\le 2|a|{q+\sqrt{q^2-4}\over q^2-4}$ (see e.g., \cite{Z2}, p.19). If $y=0$ then again $a$ must be a square and the solutions are derived from those of $(2x-qy)^2=4a$ on using the Pell equation. In fact, we have a regular section $s:\A^1\to\A^2$ of $f$, given by $s(q)=(q\sqrt a, \sqrt a)$. In particular,  we then obtain a Zariski-dense set of integer solutions; the corresponding integral values cover the whole of $\Z$. 

If $a$ is not a square, we must have $|y|\ge 1$,  whence 
$q^2\le 4+8|a||q|$, so $|q|$ is bounded and we have only finitely many integral values.
(See also \cite{Z2}, Ex. 1.19, p. 16.)  In particular, there are no regular sections  defined over $\Z$  (which may be proved directly, even over $\Q$.) 
\end{example}

\medskip

{\tt Rational points}. As before, let us look now at integral values at {\it rational} points $(x,y)\in\Q^2$. If $q$ is such a value, then the conic $x^2-qxy+y^2=az^2$ has a rational point in $\P_2(\Q)$. This exists if and only if there are points over $\R$ and all $\Q_p$. The conic is equivalent to $x^2-(q^2-4)y^2=az^2$. If $|q|\le 1$ there are real points if and only if $a\ge 0$. If $q=\pm 2$ the conic is degenerate and easy to study, so suppose $|q|>2$. There are always real points, and there are points over $\Q_p$, for all primes $p$ not dividing $2(q^2-4)a$. We may also assume $a$ squarefree.  This gives local conditions for the integer values, which we do not make explicit for brevity. Just to mention an instance, let $a=-1$, then the condition is simply that $q^2-4$ is a sum of two rational squares, which amounts to $q^2-4$ being a sum of two integral squares. In turn, this easily leads to the formulae $q={1\over 2}(d+{u^2+4\over d})$, where $u$ is any integer and $d$ is a divisor of $u^2+4$ such that $d$ and $(u^2+4)/d$ have the same parity.  In particular this easily shows that the integer values at rational points form a set which is not thin. (The same may be easily proved for any nonzero integer $a$.)

It may be also  easily proved that there are rational sections defined over $\Q$  if and only if $a$ is a square (this would amount to nonzero polynomials $x,y,z\in\Q[q]$ such that the equation is identically verified, and  one may look at leading coefficients). 

\medskip

{\tt Geometry}.  To conclude this example, let us briefly discuss its geometry, assuming $a\neq 0$. We may blow up $\P_2$ at the four points of intersection $P=Q=0$, i.e. $(0\pm\sqrt a)$, $(\pm\sqrt a,0)$, obtaining a  blown-up surface $Z$ with the usual morphism $\pi:Z\to \P_2$, and a regular map $F:Z\to \P_1$  such that $F=f\circ \pi$ out  of the four exceptional divisors.  The integral values of $f$ correspond to the integral points of $Z$ with respect to the divisor $F^*(\infty)+D$, where $D$ is the pull-back of the line at infinity on $\P_2$ if we look at integer points $(x,y)$, and $D=0$ otherwise.  On the other hand, $F^*(\infty)$ is the proper transform of the lines $xy=0$.  The canonical class $K_Z$ of $Z$ is $-3L+E$ where $E$ is the sum of the four exceptional divisors and $L$ is the pull-back of a line.  So $K_Z+F^*(\infty)+D=D-L$, which is $0$ in the first case (we obtain a so-called {\it log-K$3$ surface)}, $-L$ in the second case.  Vojta's conjectures well account for `many' integral values, at any rate over some number field and outside a suitable finite set $S$ of places.

We note that if $q$ is a value we have  the equation $f=q$, i.e. $P(x,y)=qQ(x,y)$, which defines a cubic in the variables $x,y,q$. This yields another (equivalent) viewpoint. 

This last analysis is similar for the next examples, so we shall omit it from now on.

Our next example comes as a slight generalisation of a well-known Olympiad problem (indeed, the case $a=0$ appeared in Problem 6 in the 29-th Intern. Math. Olympiad, Canberra, 1988 - see \cite{DJMP}).

\begin{example}\label{EX.2}: {\bf Divisibility of $x^2+y^2-a$ by $xy+1$}. Now we let $P(x,y)=x^2+y^2-a$, $Q(x,y)=xy+1$. Again, if $q\in\Z$ is a value we have $x^2-qxy+y^2=a+q$, or $(2x-qy)^2-(q^2-4)y^2=4(a+q)$. Suppose $|q|>2$. Then if this conic has an integral point the minimal solution has (as in the previous example) $|y|^2\le 2 |a+q|{q+\sqrt{q^2-4}\over q^2-4}$. 

 In particular, for large $|q|$ we have $|y|\le 2$. For $y=0$ we obtain the values $q=x^2-a$, $x\in\Z$, hence a thin set. For $|y|=1,2$, we obtain that $y^4-ay^2+1$ is divisible by $xy+1$, hence $|x|$ is bounded, and we obtain only finitely many values for $q$.  It follows {\it a fortiori} that the set  of corresponding integral points is thin (see especially \cite{S}, Ex. 2, p. 20). However this set is Zariski-dense, as is easy to check: for $q=m^2-a$, we have the solution $(m,0)$,  and the units $({q+\sqrt{q^2-4}\over 2})^r$, $r\in\Z$, provide an infinity of solutions with the same $m$. On varying $m,r$ we obtain a Zariski-dense set. 
 
 \smallskip
 
 $\bullet$ {\bf Olympiad problem}. The Olympiad problem referred to above is the case $a=0$ and will be relevant in \S \ref{S.simply-connected-surface} below. In this case we can be more explicit.  We find that if $|q|>2$ we have $|y|^2\le 3$ hence $|y|=0,1$ for a  solution with minimal $|y|$.  For $y=0$ (in such a solution) we find $q=x^2$, a perfect square. For $|y|=1$, we have that $x\pm 1$ divides $x^2+1$, hence  $x\pm 1$ divides $2$, and we find that the ratio $q$ must be one among $1, - 5$.

\smallskip

We leave it to the interested reader to prove that there are no rational sections (for the affine conic pencil) defined over $\Q$. 

\smallskip

{\tt Rational points}. As to integral values at {\it rational } points, things are more involved than the previous example.  As remarked in Ex. 2.5, p. 103 of \cite{Z2}, already for $a=0$  the local conditions give rise to difficult problems if we look for the exact distribution of the integral $q$ (see therein for more details). If we merely seek `many' such values $q$, a simple and elegant  argument of {\sc Heath-Brown} suffices. We reproduce it from Ex. 2.5, p. 40 of \cite{Z2} (where it appeared after kind permission of {\sc Heath-Brown}) in generalised form, suitable to show that the integral values form a set which is not thin.  

We let $x=u/v$, $y=2/v$ for integers $u,v$. We need $u^2+4-av^2=q(v^2+2u)$, so 
$u^2-2qu+4=(q+a)v^2$; we want that $v^2+2u$ divides $u^2+4-av^2$. Putting $w:=u-q$  this reads $w^2+4-av^2=q^2+qv^2$, and solving for $q$ we want 
$\Delta:=v^4+4(w^2+4-av^2)=v^4-4av^2+4w^2+16$ a square, say $=(s+2w)^2$.  This translates into $2sw=v^4-4av^2+16-s^2$. Finally, for solving this in integers we merely need that the right-hand  side $v^4-4av^2+16-s^2$ is divisible by $4s$. In turn, let $v$ be an odd integer and $s$ be a divisor of $v^4-4av^2+16$; then  $s$ must be odd and the sought divisibility holds, since $v^4-s^2$ is divisible by $4$. We can now choose freely $v$ and the divisor $s$ (for instance among the infinitely many odd prime divisors of the integers $v^4-av^2+16$, $v$ odd), obtaining (as is easy to check) a non-thin set.  (We also refer here to {\sc Coccia}'s paper \cite{Coccia} for examples with cubics.)
 \end{example}

\subsection{Proof of Theorem \ref{T.divisibility}} 
 
The idea of the proof is the following: starting from one solution $(x_0,y_0)$ as in the statement, letting $m=Q(x_0,y_0)$ and using the fact that the equation $Q(x,y)=m$ defines a hyperbola with irrational asymptotes and at least one integral points, we obtain the existence of an affine automorphism of infinite order of $A^2$ conserving such hyperbola and producing infinitely many integral points $(x,y)\in \Z^2$  with $Q(x,y)=m$. Such an automorphism has finite order modulo $m$, so one of its powers $\Phi$ induces the identity modulo $m$. This implies that the image $\Phi(x_0,y_0)=:(x_1,y_1)$ satisfies $P(x_1,y_1)\equiv P(x_0,y_0)\equiv 0 \pmod m$. Hence $(x_1,y_1)$ is also a solution to our divisibility problem. The orbit of $(x_0,y_0)$ under  $\Phi$ is infinite (because $(x_0,y_0)$ is not the centre of symmetry of the relevant hyperbola) and we obtain infinitely many solutions to the divisibility problem. The values of $q$ in such a sequence tend to infinity in view of the positiviey assumption  on the polynomial $P(x,y)$. Now, for every sufficiently large $q$, the curve of equation $P(x,y)-qQ(x,y)=0$ will also be a hyperbola and for infinitely many such $q$ its asymptotes will be irrational. Applying automorphisms of such an hyperbola, we obtain infinitely many solutions with the same values of $q$. Putting this together, we obtain a Zariski-dense set of solutions.

 To prove our assertion concerning the Hilbert Property, we follow Coccia's paper \cite{Coccia}, which in turn was inspired by our previous work \cite{CZhilb}.
 
We consider the surface $S$ obtained by blowing up the four intersection points of the conics of equation $P(x,y)=0$ and $Q(x,y)=0$ in the projective plane. It is a Del Pezzo surface of degree $5$. The divisibility problem is solved by the integral points on $S$ with respect to the union of the pull-back of the line at infinity and the  strict transform of the conic $Q(x,y)=0$. 

 We apply  Theorem 1.7 from \cite{Coccia}. The double family of rational curves on $S$ appearing in teh statement of Theorem 1.7 of \cite{Coccia} is represented  by the strict transform of the conics passing through the blown-up points and the pull-back of the pencil  of conics generated by $Q(x,y)=0$ and the double line at infinity. 
 
 We  leave to the reader to check that all the hypotheses appearing in Theorem 1.7 of \cite{Coccia} are satisfied, apart the one (unnecessary) that the intersection product between any two rational curves of the two families is $2$ (in our case it is $4$, but Coccia's proof works whenever this product is $\geq 2$). 
 
 \bigskip

\begin{example}\label{EX.3}: {\bf Divisibility of $x^2+y^2$ by $x^2-2y^2-1$}.  Theorem \ref{T.divisibility} can be applied, noticing that the hypothesis on the numarator and denominator are satisfied and the point $(1,1) $ is an admissible solution; we obtain that the solutions $(x,y)$ form a Zariski-dense non-thin set of the plane.

We want, however, to add a couple of remarks specific to this special case.

\smallskip

$\bullet$ {\tt   Some congruence conditions}. We note that we have to produce integers $q$ such that
the equation  $(q-1)x^2-(2q+1)y^2=q$ has an integral solution $x,y$.  This yields, to start with, some local conditions for solvability, on which we rapidly comment. 

Indeed,  in particular, if $p$ is an odd prime dividing $q-1$, so $q\equiv 1\pmod p$, we have, from the existence of an integral point, that $-3\equiv -q(2q+1)$ is a quadratic residue modulo $p$, hence $p\equiv 1\pmod 3$  (note that $p=3$ is impossible). Also, if $p$ divides $2q+1$ then $4q(q-1)\equiv 3\pmod p$ must be  a quadratic residue mod $p$.  Finally, if an odd  $p$ divides $q$ exactly, then $p$ cannot divide $xy$ and $-1$ is a quadratic residue, hence $p\equiv 1\pmod 4$.  So in particular each among  $q$, $q-1$ and $2q+1$ is  composed   only of special primes, which itself is a nontrivial constraint. 

Further, suppose  that $q$ is even and $>0$,  so  $q-1$ is odd and then $q-1\equiv 1\pmod 3$ by reciprocity. Also, $2q+1\equiv 1\pmod 4$ and Jacobi reciprocity yields $2q+1\equiv 1\pmod 3$, a contradiction. Hence $q$ is odd, and therefore $\equiv 1\pmod 4$. Also, $y$ must be odd, so $(2q+1)y^2+q\equiv 0\pmod 4$.

\smallskip

$\bullet$  {\tt A continued fraction}. As a final remark before the proof, we observe that a solution  (with $q>1$, $x,y>0$) yields $x^2-\delta y^2={q\over q-1}$, where $\delta ={2q+1\over q-1}$, and this entails 
$0<x-y\sqrt\delta={q\over (q-1)|x+y\sqrt\delta|}<{q\over 2(q-1)\sqrt\delta}y^{-1}$. Since $(q-1)\sqrt\delta\ge q$, it is well known that this implies that $x/y$ is a convergent to the continued fraction for $\sqrt\delta$. 

\medskip

\end{example}

\smallskip

 

\section{A simply connected affine surface without the Hilbert Property }\label{S.simply-connected-surface}
 
In this  section we prove Theorem \ref{T.dense-and-thin}. More precisely,  we use Example \ref{EX.2} to exhibit a simply connected affine surface defined over $\Q$, having a Zariski-dense set of integral points over $\Z$, but without the Hilbert Property for these integral points: indeed, such that the integral points all lie in the image of the integral points from a connected covering of degree $\geq 2$ plus a single extra point. 

In the paper \cite{CZhilb}, this issue appeared as  a special case of a general  problem  about the Hilbert Property holding or not  for simply connected varieties (which was shown to be a necessary condition). The present result  gives a negative answer, at any rate restricting to  the integral points. 

The surface is given in a very explicit way, namely as the smooth affine surface $S$ defined in $\A^3$ by the equation
\begin{equation}\label{E.S}
S:\qquad x^2+y^2=z(xy+1)\qquad\subset\quad  \A^3.
\end{equation}

However we point out that     very probably, in our view, 
the answer to the said problem is affirmative if we allow a finite extension of the ground field (i.e., the Hilbert Property over a suitable number field should hold  for simply connected varieties). This is certainly true for the present example, and is not difficult  to check: it suffices that the unit group is infinite and then one can use for example  the map $\G_{\rm m}^2\to S$ given by $(u,v)\mapsto (u, u^{-1}(v-1), v^{-1}(u^2+u^{-2}(v-1)^2)$ to construct a non-thin set.  The issue also remains to establish whether an example similar to the present one exists (over $\Q$) for a smooth {\it projective} surface. 

\medskip

Let us now go ahead, to prove Theorem \ref{T.dense-and-thin} above. 

We start with a lemma, which proves the sought topological property.

\begin{lemma}
The smooth affine surface $S$ defined in $\A^3$ by the equation \eqref{E.S}
is simply connected, that is, $S(\C)$ is a simply connected topological space.
\end{lemma}

\begin{proof}
Let $S=S(\C)\subset \A^3(\C)=\C^3$ be the surface defined by the above equation. It is certainly irreducible, and rational.

We first note that the natural completion of the affine surface under the embedding $\A^3\hookrightarrow \P_3$ would produce a singular (projective) cubic surface (with an isolated singularity at the point $(x:y:z:w)=(0:0:1:0)$.) 

 We shall work, however, with another model, viewing $S$ as an open subset of a smooth  projective surface $\tilde{S}$, which we now construct.

Let $(x:y:t)$ be homogeneous coordinates in $\P_2$. Consider the two conics of equations
\begin{equation*}
\begin{matrix}
C_1: & x^2+y^2&=&0\\
C_2: & xy-t^2 &=& 0
\end{matrix}
\end{equation*}
Note that $C_1$ is singular at the point $(0:0:1)$ and that the intersection $C_1\cap C_2$ consists of four distinct  points $P_1,\ldots,P_4$, lying outside the line at infinity $t=0$. 
 We define the surface $\tilde{S}$ to be the projective plane $\P_2$ blown up at the four points $P_1,\ldots,P_4$. 

From equation \eqref{E.S} it follows that the regular function $z$ on $S$ is rationally expressed as $z=(x^2+y^2)/(xy+1)$; viewing $z$ as a rational function   $\A^2\dashrightarrow \A^1$, its indetermination locus is the set  $C_1\cap C_2$. Hence $z$  gives rise to  a regular map  to $\P_1$ on the blow-up of the {\it affine} plane $\A^2$ over $P_1,\ldots,P_4$. 
Then $\tilde{S}$ is simply connected because it is a smooth projective rational surface.

Now, it is easy to see that $S$ is isomorphic to the complement in $\tilde{S}$ of the union of the pre-image of the line at infinity in $\P_2$ with the strict transform of the conic $C_2$. 

We want to prove that this quasi-projective (indeed affine) surface is simply connected. We follow the arguments in \S 3 of \cite{CZadv}, especially the proof of Theorem 3 therein. Viewing $S$ inside $\tilde{S}$ as above, let $X\subset S\subset \tilde{S}$ be  the surface obtained by removing from $S$ the exceptional divisors lying above the points $P_1,\ldots,P_4$, hence $X$ is the affine plane $\A^2$ minus the conic $C_2$. As shown e.g. in Lemma 2 in \cite{CZadv}, the inclusion $X\hookrightarrow S$ induces a surjective homomorphism $\pi_1(X)\rightarrow \pi_1(S)$. Now, since $X$ is isomorphic to the complement of a line and a smooth conic in general position, its fundamental group is infinite cyclic, according to a theorem of Zariski (see e.g. \cite{SeBourbaki}). 

It follows that the fundamental group of $S$ is also cyclic. To conclude, it suffices to exclude the existence of non-trivial cyclic covering of $S$.

 Now, finite cyclic covers of degree $p$ (say $p$ prime, for simplicity) are associated to rational functions $f$ whose principal  divisor  is of the form $(f)=pD$, where $D$ is a non-principal  divisor.  Since $\tilde{S}$ is simply connected, any unramified cover of $S$   of degree $>1$ would come from a cover of $\tilde{S}$ ramified   above (at least one of)  the two components at infinity, say $A$ and $B$, where $A$ is the pre-image of the line at infinity and $B$ the strict transform of the conic $C_2$. Hence we are looking for a prime $p$, a rational function $f\in \C(\tilde{S})$, a divisor $D$ on $S$ and two integers $a,b$, not both divisible by $p$, such that
\begin{equation*}
(f)=pD+aA+bB
\end{equation*}
Recall that the four exceptional divisors $E_1,\ldots,E_4$ on $\tilde{S}$   over the points $P_1,\ldots,P_4$ satisfy (for each $i=1,\ldots,4$) $E_i\cdot A=0$, $E_i\cdot B=1$. Since $(f)\cdot E_i=0$, it follows that $p|b$. Since also $(f)\cdot A=0$ and $A^2=1$, it follows that $p|a$, which contradicts the fact that not both $a$ and $b$ can be divisible by $p$.

\end{proof}
\medskip

\begin{proof}[Proof of Theorem \ref{T.dense-and-thin}]  We take $S$ as the surface  given by \eqref{E.S},  which  is simply connected by the lemma. 

To prove the rest, we invoke   Example \ref{EX.2} above, telling  that the set $S(\Z)$ of  integral points  of $S$ is  Zariski-dense and such that  for each $(x,y,z)\in S(\Z)$ either $z=-5$ or  $z$ is  a perfect square (this second conclusion, holding for all positive choices  of $x,y$, is the content of the mentioned IMO Problem (see \cite {DJMP})). Hence all   integral points but those in the    curve  determined in $S$ by $z=-5$,  lie in the image of the set $X(\Z)$ of integral points from  the double cover $X$ of $S$ given by $z=w^2$ (defined in $\A^3$ by $x^2+y^2=w^2(xy+1)$).
The cover is clearly irreducible so this situation violates the Hilbert Property (over $\Q$)  for the points in $S(\Z)$. 
  \end{proof}
 
The main difference between this example and the situation considered in Theorem \ref{T.divisibility} consists in the fact that the zero set of the denominator in the expression $(x^2+y^2)/(xy+1)$ is a hyperbola with {\it rational} asymptotes.  This fact prevents the application of the double conic fibration method  which would provide the Hilbert Property.

\section{Common values of rational functions} 
In this section we consider {\it common values} of two rational functions at integral or rational points. Namely, given two algebraic varieties $X,Y$ defined over a number field $k$ and two rational functions $f\in k(X)$, $g\in k(Y)$. we shall be interested in cases where $f(X(k))=g(Y(k))$, or one inclusion bewteen such sets holds, or at least $f(X(k))\cap g(Y(k))$ is infinite.

\smallskip

The case when $X=Y=\P_1$, or $X=Y=\A^1$ and the functions $f,g$ are polynomials, has been extensively studied, in particular in works of {\sc Bilu-Tichy} \cite{BT} and of {\sc Avanzi } and the second author \cite{AZ1}, \cite{AZ2}. More recent work is due to {\sc Zieve} (unpublished) and to {\sc Behanjaina, K\"onig, Neftin} \cite{BKN} (solving in particular an old question of {\sc Davenport-Lewis-Schinzel}).

 In these cases, it can be shown  that the inclusion $f(X(k))\subset g(Y(k))$ leads to the existence of  rational functions $\phi,\psi \in k(t)$ such that $f=\phi\circ g\circ \psi$. 
The classification of pairs of rational functions $f,g$ such that the intersections $f(\P_1(k))\cap g(\P_1(k))$ is infinite leads to delicate geometric problems on the irreducibility of the algebraic curves of equation $(f(x)-g(y))$ or $(f(x)-f(y))/(x-y)=0$. 
 
\smallskip

We shall concentrate now mainly on the case $X=Y=\A^2$.  We then choose two rational functions $f(x,y), g(z,w)$ in two variables and look for common values of these functions at integral or rational points. 
While the case in which one of the two functions has degree $1$ in either $x$ or $y$ is trivial for rational points, in all other cases we are led to  interesting problems in Diophantine geometry. 

\medskip

We restrict our attention to the case of {\it homogeneous} functions, namely quotients of homogeneous forms; also, we restrict to the case where the degrees of the two rational functions are the same. 

When such  degree is zero, the   problem reduces to the above one with functions of one variable; as mentioned, this case has been treated in \cite{BT}, \cite{AZ1}, \cite{AZ2}.

\subsection{Common values of homogeneous functions of degree 1}

Then,  the first  problem we shall consider here arises when these rational functions have degree $1$; the most basic case is when  one of the form, say the numerator, has degree $2$ and the other has degree $1$. We shall then study equations of the form
\begin{equation}\label{E.common-values1}
\frac{q_1(x,y)}{l_1(x,y)}=\frac{q_2(z,w)}{l_2(z,w)}
\end{equation}
where $q_1,q_2$ are quadratic forms and $l_1,l_2$ are linear. We are looking for  pairs of vectors $(x,y),(z,w)\in\Z^2$ satisfying the above equation. We are also interested in finding pairs of {\it primitive} vectors satisfying the same equation. Theorem \ref{T.Cayley} below  will show that the distribution of such pairs of integer vectors is different in the two cases. 

We shall  see later how this problem is connected to  a certain version of the Dynamical Mordell-Lang conjecture (see the book \cite{BGT} for the general setting and numerous instances and see the  authors' article \cite{CZbilliards} for the specific case in question). 

\medskip

{\tt The geometry of the Cayley surface}. The equation \eqref{E.common-values1} can also be written in polynomial form as the equation of a cubic surface in $\P_3$ (interpreting $x,y,z,w$ as homogeneous coordinates):
\begin{equation}\label{E.Cayley}
l_2(z,w)\cdot q_1(x,y)=l_1(x,y)\cdot q_2(z,w).
\end{equation}
For general choices of $q_1,q_2,l_1,l_2$ we obtain the so-called {\it Cayley's nodal cubic surface}, i.e. a cubic surface having four isolated singularities. Letting $(\alpha_1:\beta_1), (\alpha_1':\beta_1')$ (resp. $(\alpha_2:\beta_2)$, $(\alpha_2':\beta_2')$) the zeros of the quadratic form $q_1(x,y)$ (resp. of $q_2(z,w)$), these singular points points are $(\alpha_1:\beta_1:0:0),  (\alpha_1':\beta_1':0:0)$, $(0:0:\alpha_2:\beta_2)$ and $(0:0:\alpha_2':\beta_2')$. The genericity conditions alluded to above, ensuring that there will be no other singularities, are: (1) the quadratic forms $q_1,q_2$ are non-degenerate and (2) they are not divisible by the corresponding linear forms. 

The desingularization of the above surface can be obtained as the blow-up of the plane over the six vertices of a complete quadrilateral: namely, starting from four lines in general position on $\P_2$, blow up the six intersection points among the six pairs of the four lines; the linear system of cubics passing through the six points provides a rational map $\P_2\to\P_3$, regular on the blown-up surface,  whose image is a cubic surface; the map  contracts the four given lines to four singular points. It is clear from this description that all these surfaces are pairwise isomorphic over $\C$. Such singular cubic surfaces contain nine lines, instead of the twenty seven lying on any smooth cubic surface. The surfaces defined over $\Q$ by an equation of the form \eqref{E.Cayley} always contain at least three lines defined over $\Q$ (as shown in the proof of the following theorem).

\smallskip

The general theorem we can prove concerning integral solutions of equation \eqref{E.common-values1}  reads

\begin{theorem}\label{T.Cayley}
Let $q_1(x,y),q_2(z,w)$ be quadratic forms with integral coefficients and $l_1(x,y),l_2(z,w)$ be linear forms with integral coefficients such that the surface of equation \eqref{E.Cayley} admits only four singular points. Then   the integral solutions $(x,y,z,w)\in\Z^4$ to the equation \eqref{E.common-values1}   form a  Zariski-dense set of rational points in the variety defined by \eqref{E.common-values1}.

On the contrary, the set of pairs of primitive vectors $(x,y), (z,w)\in\Z^2$ which solve equation \eqref{E.common-values1} define a subset of rational points on the surface which is contained in finitely many conics.
\end{theorem}

We note that by a solution to equation \eqref{E.common-values1} we automatically mean a solution with $l_1(x,y), l_2(z,w)\neq 0 $. 

\smallskip

\begin{proof}
Concerning the first part of the theorem, note that the integral solutions to equation \eqref{E.common-values1} correspond to rational points on the cubic surface of equation \eqref{E.Cayley}. Such a surface  always contains at least three lines defined over $\Q$, namely the lines $r_1,r_2,s$ of equation 
$$
r_1:\, x=y=0, \qquad  r_2:\, z=w=0, \qquad s:\, l_1(x,y)=l_2(z,w)=0.
$$
Note that $r_1,r_2$ are  skew lines. Each of the two  lines $r_1,r_2$ defines a pencil of planes on $\P_3$, hence a pencil of conics on the surface. Every rational point $P$ on $s$, outside $r_1\cup r_2$,   is contained in two smooth conics, one on the plane generated by $P$ and $r_1$, one on the plane generated by $P$ and $r_2$. Each such conic admits then infinitely many rational points.  Moving the point $P$ on $s$ we obtain a Zariski-dense set of rational points on the cubic surface, hence the first conclusion of the theorem.

\smallskip

Concerning the second conclusion, we note that the primitivity of the vector $(x,y)$ amounts to the integrality of the corresponding rational point $(x:y:z:w)$ with respect to the line $r_1$, while the primitivity of the vector $(z,w)$ corresponds to integrality of the same rational point with respect to $r_2$.

Now, the second conclusion of our theorem states that such integral points cannot be dense on the surface defined by equation \eqref{E.Cayley}, and more precisely they lie on a finite union of  conics contained on that surface.

This fact can be proved via a Runge-like argument, by observing that the rational function $\varphi=l_1(x,y)/l_2(z,w)$, when viewed on the desingularitazion of the cubic surface, has a pole on the divisor $r_2$ and a zero on the divisor $r_1$. Hence if we calculate $\varphi$ at  rational points which are integral with respect to $r_1+r_2$, we expect these values to be units in $\Z$; actually, their numerator and denominators can be divisible by primes of bad reduction  (where the multiplicity are also bounded);  then $\varphi$ takes only finitely many values at such points,  so these points lie on finitely many fibers of $\varphi$, which correspond to conics on the surface.
\end{proof}
\medskip

Let us  analyze the following specific example, giving an idea of the general situation: 

\begin{example}\label{Ex.Cayley}
Consider the two rational functions 
$$
f(x,y)=\frac{x^2-2y^2}{x+y}, \quad g(z,w)=\frac{z^2-3w^2}{z+w}.
$$
{\it Then the set of lattice points $(x,y)\in\Z^2$ such that there exists a lattice point $(z,w)\in\Z^2$  with  $f(x,y)=g(z,w)$ is Zariski-dense on the plane.} On the contrary: {\it 
The set of pairs primitive vectors $(x,y), (z,w) \in \Z^2$ for which   $f(x,y)=g(z,w)$ is not Zariski-dense and is parametrized as
\begin{equation}\label{E.parametrization}
(x:y:z:w)=(2s^2+6ts+t^2:-2ts:4s^2+6ts+t^2:-2s^2-2ts)
\end{equation}
where $(s,t)\in\Z^2$ are coprime integers, with $t$ odd.}
 \end{example}

Note that although the solutions in pairs  of primitive vectors $(x,y),(z,w))$ to equation \eqref{E.Cayley} do not generate a Zariski-dense set on the surface, nevertheless the set of primitive vectors $(x,y)\in\Z^2$ for which there exists a primitive vector $(z,w)\in\Z^2$ such that $f(x,y)=g(x,y)$ is Zariski-dense in $\A^2$. 

\medskip

{\tt Application to the Dynamical Mordell-Lang Conjecture}. We now apply the results of Example \ref{Ex.Cayley} to a question concerning the dynamical Mordell-Lang conjecture.

 In \cite{CZbilliards} we proposed to study the following version of the Dynamical Mordell-Lang Conjecture: suppose given an algebraic surface, three algebraic curves on it  and a {\it commutative} semigroup of endomorphisms. What can be said in  case infinitely many orbits under the semigroup intersect all the three curves?

We proved some finiteness results, in some case interpreting them in terms of  trajectories on elliptic billiards. We also discussed about the condition of commutativity of the endomorphism semigroup.

Here we shall  show, by treating a specific case,  how  in the non commutative case the situation can be very different. Let us consider  the usual linear action of the group $\mathrm{SL}_2(\Z)$ on the plane $\A^2$. Consider the line $x=0$ and two conics passing through the origin, say those of equations
$$
x^2-2y^2=x+y,\qquad \mathrm{and}\qquad x^2-3y^2=x+y.
$$
 We claim that infinitely many orbits of points of $\A^2$ intersect all the three curves. We start from a solution $(x,y,z,w)$ of the equation $f(x,y)=g(z,w)$, where $f,g$ are as in the previous theorem, with $(x,y)$ and $(z,w)$ primitive vectors of $\Z^2$; as we have just proved, there are infinitely many of them.  
Let $\lambda$ be the common value of $f(x,y)=(x^2-2y^2)/(x+y)$ and $g(z,w)=(z^2-3w^2)/(z+w)$, so $\lambda=(x^2+y^2)/(x+y)$ and $\lambda=(z^2-3w^2)/(z+w)$. The primitive vector $(x,y)$ can be completed to a unimodular matrix 
$$\left(\begin{matrix}x'& x\\ y' & y\end{matrix}\right)\in\mathrm{SL}_2(\Z)$$
 while the vector $(z,w)$ can be completed to a unimodular matrix $\left(\begin{matrix}z'& z\\ w' & w\end{matrix}\right)\in\mathrm{SL}_2(\Z)$.
 Then the point ${0\choose \lambda^{-1}}$,  lying on the line $x=0$ , is sent by multiplication by the first matrix to the point ${\lambda^{-1}x \choose \lambda^{-1}y}:={u\choose v}$ which satisfies the equation $u^2-2v^2=u+v$ of the first conic. Analogously, the second matrix sends the same point to a point of the second conic. Hence the points of the form ${0\choose \lambda^{-1}}$, with $\lambda$ as above, have an orbit intersecting each of the three curves.
 
 Note that by the parametrization \eqref{E.parametrization}, the set of such values of $\lambda$ is infinite, and so the corresponding points ${0\choose \lambda^{-1}}$ generate infinitely many orbits under the action of $\mathrm{SL}_2(\Z)$. 

\medskip

\subsection{Common values of homogenous functions of degree $2$}
The second case to be considered arises when the numerators have degree three and the denominator degree one. We obtain again an equation of the form \eqref{E.Cayley}, but now the homogeneous forms $q_1(x,y), q_2(x,y)$ are cubic forms, so the defined surface has degree $4$. 

Under generic assumptions, i.e. both cubic forms have distinct linear factors (in a factorization over the complex number fields) and the linear forms do not divide the corresponding cubic forms, this quartic surface admits exactly six singular points, lying on the union of three lines, three on each line. As for the Cayley cubic surface, the quartic surfaces of this kind are all pairwise isomorphic over the complex; they have been studied by Weddle in 1850, and are named after him.

The arithmetical result we can prove reads as follows:

\begin{theorem}\label{T.Weddle}
For suitable choices of linear forms $l_1,l_2$ and cubic forms $q_1,q_2$, the integral solutions to equation \eqref{E.Cayley} are Zariski-dense in the quartic surface defined by \eqref{E.Cayley}. On the contrary, the solutions in pairs of primitive vectors $(x,y), (z,w)$ lie in the union of finitely many curves of equation $l_1(x,y)=\lambda l_2(z,w)$ on the quartic surface.
\end{theorem}

 {\it Sketch of the proof}.  A Zariski-dense set of rational points can be constructed as follows: consider, as before,  the skew lines  in $\P_3$ of equation 
$$
r_1: x=y=0,\qquad r_2: z=w=0.
$$ 
Each of them defines a pencil of planes whose intersections with the surface determines a pencil of cubic curves of genus one (it turns out that these curves are all twists of the Fermat cubic curve). The intersection of the generic plane passing through $r_1$ with $r_2$ determines a section of the first fibration and viceversa. Taking these sections for the zero-sections, we obtain a double elliptic fibration on the quartic surface $S$ of equation  \eqref{E.Cayley}. A generic point $P\in S  \setminus (r_1\cup r_2)$ lies in two elliptic surfaces (i.e. two smooth fibers, one for each fibration). Suppose $P$ is a rational point; unless $P$ is torsion for both elliptic curves,  one of the two fibers will contain infinitely many rational points. By a celebrated theorem of Silverman, only finitely many of them can be torsion for the second fibration; we then obtain infinitely many fibers each possessing infinitely many rational points, thus providing the Zariski-density of rational points. 

Of course we need a rational point $P\in S(\Q)$ outside the union of $r_1,r_2$ to start with; the points on the line $l_1(x,y)=l_2(z,w)=0$ do not work, since the corresponding fibers are both singular.

\smallskip

Concerning solutions with primitive vectors $(x,y)$ and $(z,w)$, the same integrality considerations used in the case of the Cayley cubic surface apply also to the Weddle quartic surface: a Runge-like argument shows that all integral points are contained in finitely many curves of equation $l_1(x,y)=\lambda l_2(z,w)$; however, now these curves now are (generically) genus one curves instead of conics. Again, there might be infinitely many such points.

\subsection{Pairs of lattices with vectors of equal length} 
Let us consider the following problem in plane geometry: \smallskip

{\it Problem: given two lattices $\Lambda, M\subset\R^2$, suppose that each circle centred in the origin intersects $\Lambda$ if and only if it intersects $M$. Is it true that $\Lambda,M$ are isometric?}

\smallskip

The answer is `almost' yes: Delaunay proved there is just one exception, up to obvious actions of the group of isometries and the group of homotheties, given by the hexagonal lattice $\Lambda$ and any of its index-two sublattices $M$. 

However, the problem changes is we consider only primitive vectors. In Delaunay's example,  it is not true that the circles intersecting $\Lambda$ in a primitive vector are the same circles intersecting $M$ at a primitive vector.  Indeed, if $\Lambda$ is realized as the ring of integers $\Z[\zeta]$ of the sixth cyclotomic field and $M$ the subgroup generated by $\zeta$ and $2$, the vectors of length $2$ on $\Lambda$ are not primitive, since the equation $x^2+xy+y^2=4$ does not admit primitive solutions $(x,y)\in\Z^2$), while the vector $2\in M$ is primitive as a vector of $M$.

\smallskip

\smallskip

We shall prove the following: 

\begin{theorem}\label{T.lattices}
Let $\Lambda, M\subset\R^2$ be two lattices such that the set of lengths of primitive vectors for the two lattices are the same up to possibly finitely many exceptions. Then $\Lambda$ is isometric to $M$. 
\end{theorem}

\begin{proof} 

Let us consider first the case when the two values sets coincide exactly, i.e. the possible lengths of primitive vectors of $\Lambda$ and of $M$ are exactly the same, then  in particular the two lattices share the first three minima. In this case, it is known that they are isometric (see e.g. the ``Well lemma'' in {\sc Conway}'s monograph \cite{Conway}).

Suppose now that the two sets of values have finite symmetric difference. We prove that indeed they must coincide, so we can apply the previous argument based on the first minima.

We easily reduce to the case when all square lengths are integral (or rational, which amounts to the same up to applying a suitable homothety to both lattices (indeed,  when the vector space over $\Q$ generated by the square lengths has dimension $\geq 2$ the problem is easier, since we obtain more quadratic forms sharing the same values at integral points).

We can then suppose that $q_1(x,y)$ and $q_2(x,y)$ are integral valued definite quadratic forms such that the sets of their values at primitive lattices differ by a finite set.

We can also suppose the two quadratic forms are expressed by   {\it primitive} homogeneous polynomials, so that for each of the two forms no prime number divides all its values. 

It is well known that both quadratic forms represent infinitely many prime numbers (and necessarily they represent these numbers at primitive vectors).

We can find orders $\mathcal{O}_1,\mathcal{O}_2$ in rings of imaginary quadratic fields and ideals $I_1\subset \mathcal{O}_1$, $I_2\subset\mathcal{O}_2$, such that for $j=1,2$ the form $q_j$ represents the ratios $N((\alpha))/N(I_j)$, for $\alpha\in I_j$, i.e. the indices of principal ideals inside $I$. Tensoring with $I_j^{-1}$, one can view the values of $q_j$ as the norms of those ideals of $\mathcal{O}_j$ which belong to the class  of $I_j^{-1}$. 

 Let us first show that the quadratic number rings containing these orders are the same. Otherwise, we could find infinitely many rational primes splitting in the first number field and not in the second one. No product of such primes could be primitively represented by the second quadratic form; on the contrary, given a square-free number  $k>1$ represented by the first quadratic form,  factoring $k$ as $k=p_1\cdots p_t$ we find that the product $p_1\cdots p_t$ is representable. Each prime $p_i$ factors in the ring of integers of the first number field as $\mathcal{P}_i\overline{\mathcal{P}_i}$; taking other prime ideals $\mathcal{Q}_i$ in the same class of $\mathcal{P}_i$, which is possible by Dirichlet's theorem, we find that the norm of $\mathcal{Q}_1\cdots \mathcal{Q}_t$ is representable by the first form, not by the second, and this can be done for infinitely many choices of prime ideals.
 
We can then suppose that the two orders are contained in the same quadratic number field.
Let us now extend the previous argument to show that indeed the two orders $\mathcal{O}_1,\mathcal{O}_2$ coincide and the two ideals $I_1,I_2$ lie in the same class. 

Suppose that a number $k>0$, coprime with both the conductor of $\mathcal{O}_1$ and that of $\mathcal{O}_2$,  is represented by $q_1$ and is not represented by $q_2$. This means that there is an ideal $J$ in the class of $I_1^{-1}$ of norm $k$, but no ideal of such a norm  in the second class. We shall produce infinitely many other integers $k'$ which are represented bt the first form and not by the second.

Let $\mathcal{O}=\mathcal{O}_1\cap \mathcal{O}_2$ be the intersection of the the orders.

Let us factor $J$ into prime ideals as $J=\mathcal{P}_1\cdots \mathcal{P}_t$, we have $k=N(\mathcal{P}_1)\cdots N(\mathcal{}_t)$.  
Let $(\alpha)\subset\mathcal{Ø}$ be a principal prime ideal of $\mathcal{O}$, coprime with $J$.  Then the ideal $J'(\alpha)\cdot \mathcal{P}_1\cdots \mathcal{P}_t$ is in the same class as $J$, so its norm $k':=N(J')=N(\alpha)\cdot N(J)$ represented (primitively) by $q_1$, and for the same reason it is not represented by $q_2$. Since we can find infinitely many such elements $\alpha$, we find infinitely many integers represented by $q_1$ and not represented by $q_2$.
\smallskip

\end{proof}

We end by showing that it may happen that two non-equivalent quadratic forms represent exactly the same primes, as in the following example.

\begin{example}\label{Ex.same-primes}
Consider the two quadratic forms 
$$
q_1(x,y)= x^2+9y^2,\qquad q_2(x,y)=x^2+12y^2.
$$
The first one corresponds to a sub-lattice of index $3$ in the square lattice $\Z[i]$, the second one to an index $6$ sub-lattice of the exagonal lattice $\Z[\zeta]$. They represent the same (infinite) set of primes, namely the primes congruent to $1$ modulo $12$. Indeed, the quadratic form $q_1(x,y)=x^2+(3y)^2$ represents only primes congruent to $1$ $\pmod 4$ and each number represented, if not divisible by $3$, must be congruent to $1$ $\pmod 3$. Also, a prime which is congruent to $1$ modulo $12$ must be of the form $u^2+v^2$ where one between $u$ and $v$ is divisible by $3$. Hence all such prime numbers are represented by $q_1$.

Analogously, every prime number congruent to $1$ $\pmod 3$ is of the form $u^2+3v^2$, and, if it is also congruent to $1$ $\pmod 4$, $v$ must be even.

However, the two quadratic forms do not share the same values at primitive vectors: for instance, the first quadratic form represents all the products of the form $pq$, where $p,q$ are distinct primes congruent to $5$ modulo $ {  12}$, while the second one does not. On the contrary, the second quadratic forms represent all the product $pq$, for distinct prime numbers congruent to $7$ modulo $12$.

Hence each difference  of their value sets is infinite.

This example shows that two  distinct value sets can have the same  intersection with the set of prime numbers, such an intersection being infinite, without coinciding.
\end{example}

We remark that pairs of binary quadratic forms representing ``almost'' the same primes have been classified (see \cite{Voight}). By a case-by-case analysis one could reprove  Theorem \ref{T.lattices}, verifying that for each of these pairs  the correpsonding value sets have infinite symmetric difference.

\smallskip

Let us now show how the above Theorem \ref{T.lattices} fits in the frame of classifying common values of rational functions on surfaces. Let the two lattices $\Lambda, M$ appearing in Theorem \ref{T.lattices} be represented by two positive definite quadratic forms $q_1,q_2$. 

Consider the surface in $\P_3$ of equation 
$$
q_1(x,y)=q_2(z,w).
$$
It is an irreducible quadric surface, bi-ruled over $\R$ but not necessarily over $\Q$. It is endowed with two rational maps projecting $(x:y:z:w)\mapsto (x:y)\in\P_1$ and $(x:y:z:w)\mapsto (z:w)\in\P_1$; each map is undefined over a pair of (complex conjugate) points. Normalizing this map we obtain a smooth projective surface $\tilde{S}$, birational to the quadric, on which the projections to $(x:y)$ and to $(z:w)$ are well-defined. We then obtain two morphisms 
$$
\pi_1: \tilde{S}\to \P_{1} \qquad  \pi_2: \tilde{S}\to \P_{1},
$$
which provide the surface  $\tilde{S}$ with two distinct fibrations in conics.

The pairs of primitive vectors with equal length correspond to the rational points on $\tilde{S}$ which are integral with respect to the divisors lying above the four indeterminacy points of the two projections (which, viewed in $\tilde{S}$, become rational curves). Let $S$ be the open surface obtained by removing these curves. Then we are interested in the set $S(\Z)$ of integral points on such a quasi-projective surface.

Theorem \ref{T.lattices}  states that not both projections $S(\Z)\to \P_1(\Q)=\P_1(\Z)$ can be surjective, even up to finite sets, unless the two lattices are isometric. In that case, there are sections $\sigma_1:\P_1\to S$ of $\pi_1$ and   $\sigma_2:\P_2\to S$ of $\pi_1$ such that the morphisms $\pi_2\circ\sigma_1$, $\pi_1\circ \sigma_2$ are isomorphisms $\P_1\to\P_1$.

\medskip

Let us now show yet another `geometric' viewpoint on Theorem \ref{T.lattices}. Consider the quasi-projective variety $X:=\A^2\setminus\{0\}$, which can be viewed as a complement of a divisor in the projective surface obtained by blowin-up the point $(0:0:1)$ in $\P_2$: namely, removing from that surface the  line at infinity plus the exceptional divisor one obtains the quasi projective variety $X$. The set of integral points $X(\Z)$ consists exactely  of the set of primitive vectors of $\Z^2$. Every polynomial in two variables with integral coefficients provides a morphism $X\to \A^1$ over $\Z$. Theorem \ref{T.lattices} shows that the value sets of two quadratic forms $q_1,q_2$, viewed as maps $X(\Z)\to \A^1(\Z)$, have infinite symmetric difference unless they are related by an automorphism of $X$.
 Example \ref{Ex.same-primes} shows  however that such sets might contain exactly the same primes,  still having infinite symmetric difference.

\smallskip

\bigskip

\section*{APPENDIX 1 - By David Masser}

Here we give a variant of Theorem \ref{T.lower-bound-heights-1} above, with a different kind of lower bound, which is however effective.

\bigskip
\noindent
{\bf Theorem}. {\it Let $\Gamma$ be a finitely generated subgroup of ${\bf G}_{\rm m}^2(\overline{\bf Q})$ and let $R(u,v)$ be in $\overline{\bf Q}(u,v)$ not of the form $S(\phi(u,v))$ for any $S$ in $\overline{\bf Q}(z)$ and any group homomorphism $\phi$ from ${\bf G}_{\rm m}^2$ to ${\bf G}_{\rm m}$. Then there are effective constants $C > 0$ and $C_0$, and an effective finite set $Z$ of
translates of connected algebraic subgroups $H \neq {\bf G}_{\rm m}^2$ by elements of ${\bf G}_{\rm m}^2(\overline{\bf Q})$, all depending only on $\Gamma$ and $R$, such that $R(p)$ is well-defined and
$$h(R(p)) \geq C^{-1}h(p)-C_0$$
for all $p$ in $\Gamma$ outside $Z$.}
\bigskip

For a comparison write $R = f/g$ for $f,g$ coprime in $\overline{\bf Q}[u,v]$. Clearly there is effective
$c$ such that the (affine) height
$$h(f(p),g(p)) \leq c(h(p) + 1)$$
and so we get a similar lower bound for $h(R(p))$ in terms of $h(f(p),g(p))$ of the same shape
as in Theorem 3.1; but with a extra multiplying constant $C^{-1}c^{-1}$ which is effective.

Our condition $R \neq  S(\phi)$ seems more natural than the condition that $f,g$ do not both
vanish at $(0,0)$. For $R = S(\phi)$, at least for surjective $\phi$, reduces the problem to a single
variable, much more easily handled; and whether this happens can be easily determined.
For example, we may assume $\phi(u,v) = u^{-m}v^n$ for $m,n$ coprime, and then the existence of
$S$ is equivalent to the differential equation
\begin{equation}\label{de}
nu{\partial R \over \partial u} + mv{\partial R \over \partial v}=0. 
\end{equation}
Furthermore it is not hard to bound $m,n$ in terms of $R$; our proof below gives
$$\max\{|m|,|n|\} \leq \max\{\deg f,\deg g\}.$$
Thus we can deal with things like
$$R(u,v) = {u + v^2 \over u^2 + v},$$
where (\ref{de}) is ruled out thanks to
$${u\partial R/\partial u \over v\partial R/\partial v}=-{u(2uv^2 + u^2- v) \over v(2u^2v-u + v^2)}$$
being non-constant.

\noindent
{\it Proof}. As above, write $R = f/g$ for $f,g$ coprime in $\overline{\bf Q}[u,v]$, of course with $g \neq 0$. We assume until the end even
\begin{equation}\label{gp}
g(p) \neq  0;
\end{equation}
it will then be easy to lift this restriction. Write also $\rho= R(p) = f(p)/g(p)$. Then we have
\begin{equation}\label{qrp}
q_\rho(p) = 0
\end{equation}
for
$$q_w(u,v) = f(u,v)-wg(u,v).$$
Note that $q_\rho$ is not constant, else (\ref{qrp}) would imply $q_\rho= 0$ and so $R = \rho= S(\phi)$ which we have excluded.

There is an absolutely irreducible factor $q$ of $q_\rho$ with 
$$q(p) = 0,$$
and by well-known estimates we can assume that the naturally defined polynomial heights $h(q),h(q_\rho)$
satisfy
$$h(q) \leq ch(q_\rho) \leq c(h(\rho) + 1)$$
where from now on $c,c,...$ denote positive constants depending only on $\Gamma$ and $R$.

It is now a problem involving effective Mordell-Lang for ${\bf G}_{\rm m}^2$, in particular with respect
to the quantity $\rho$. Assume for the moment that $q$ does not have the form $\alpha u^m-\beta v^n$ or $\alpha u^mv^n-\beta$ for $\alpha,\beta$ in $\overline{\bf Q}$ and integers $m \geq 0, n \geq 0$. Then Theorem 2.1 of B\'erczes, Evertse, Gy\"ory, Pontreau \cite{BEGP} (p.72) gives 
$$h(p) \leq  c\max\{1,h(q)\} \leq c(h(\rho) + 1)$$
and we are done (the arguments of \cite{BEGP} p.79 appear also in Bombieri, Gubler \cite{BG} p.147 and even in Bilu \cite{Bilu} p.241 with a slightly different form).

What if $q = \alpha u^m-\beta v^n$ say? Then  $\alpha,\beta$ are non-zero and $m,n$ are coprime, at
most $\deg q \leq \max\{\deg f, \deg g\}$. In particular they are bounded independently of $\rho$. An
analogous independence property for $\alpha,\beta$ is not quite so clear. Anyway, we may suppose
$n \geq 1$, and then we have
$$q_\rho(t^n,\gamma t^m) = 0$$
identically in $t$ for any $\gamma$ with $\gamma^n=\alpha/\beta$. Writing
$$f(u,v) = \sum_{i,j}f_{ij}u^iv^j,    ~~~g(u,v) = \sum_{i,j}g_{ij}u^iv^j$$
we get
$$f_k(\gamma)=\rho g_k(\gamma)$$
for all $k$, where
$$f_k(z) = \sum_{ni+mj=k}f_{ij}z^j,    ~~~g_k(z) = \sum_{ni+mj=k}g_{ij}z^j$$
Eliminating $\rho$ we get
\begin{equation}\label{fg}
f_k(\gamma)g_l(\gamma)=f_l(\gamma)g_k(\gamma)
\end{equation}
for all $k,l$.

If (\ref{fg}) do not define a finite set of $\gamma$ then all $f_k(z)g_l(z) = f_l(z)g_k(z)$ identically. Now
not all $g_k(z)$ are identically zero, otherwise so would $g(t^n,zt^m) = \sum_k g_k(z)t^k$ be, and so
$g(u,v)$ itself. Thus all the $f_k(z)/g_k(z)$ with $g_k \neq 0$ are the same; call their common value
$x = S_0(z)$. So $f_k(z) = xg_k(z)$ for all $k$. It follows now that $q_x(t^n,zt^m) = 0$ identically in
$t,z$. This is the same as
$$R(t^n, zt^m) = S_0(z).$$
Choosing $t = \zeta u^{1/n}, z = vt^{-m}$ gives
$$R(u,v) = S_0(\zeta^{-m}(u^{-m}v^n)^{1/n})$$
for all $\zeta$ with $\zeta^n = 1$. Averaging over $\zeta$ we see that $R(u,v) = S(\phi(u,v))$ for $\phi(u,v) = u^{-m}v^n$, which we also excluded.

Thus (\ref{fg}) do indeed define a finite set of $\gamma$, clearly effective. So we get at most finitely
possibilities for $\alpha/\beta$ in $\alpha u^m-\beta v^n$. Therefore we only have to exclude $p$ on the cosets $(\alpha/\beta)u^m=v^n$.

The other $\alpha u^mv^n-\beta$ are similarly handled, and we get a first approximation to our finite set $Z$.

Finally the assumption (\ref{gp}) can be justified by applying the above arguments to $g$
instead of $f-\rho g$; this time it is clear that there are at most finitely many $(\alpha/\beta)u^m-v^n, (\alpha/\beta)u^mv^n-1$ and we end up with our desired $Z$. 

Our $C$, although effective, may well be rather large, due to the use of linear forms in
logarithms in \cite{BEGP} (p.76). For example we calculated
$$h\left({3^m-1 \over 2^n-1}\right) \geq 10^{-32}\max\{m,n\}$$
for all non-negative integers $m,n$ with $n \geq 1$.

Apparently effective Mordell-Lang for the square of an elliptic curve is even now still
not known; thus we cannot prove an analogue of the Theorem for $E^2(\overline{\bf Q})$.

\vglue 2cm

\pagebreak

\section*{APPENDIX 2}

Let $a>b>1$ be integers, and assume that $a$ is not a power of $b$.  We want to prove that, under certain special {\it effective} additional hypotheses, for an integer $n\ge 1$,  the ratio ${a^n-1\over b^n-1}$ cannot be integral. 
Of course we know from \cite{BCZ}  that this integrality can hold only for finitely many $n$ (once $a,b$ are given), but that result is not effective.  In the case $a=3,b=2$, {\sc Stoll}, using suitably Jacobi reciprocity, proved   that the ratio is integral only for $n=1$. 

His argument \footnote{personal communication}  goes as follows: since $2^2\equiv 1\pmod 3$, every such $n$ must be odd. Hence the order of $3$ modulo every prime divisor  $p$ of $2^n-1$ is odd, so $3$ is a quadratic residue modulo $p$ and the quadratic symbol $({3\over 2^n-1})=1$. By Jacobi reciprocity if $n>1$ this yields $({2^n-1\over 3})=-1$. However since $n$ is odd we have $2^n-1\equiv 1\pmod 3$, a contradiction.

This nice argument  
  generalizes to other pairs $a,b$; however they will have to  satisfy certain   conditions \footnote{For instance it suffices  that $b$ is even,  with $\gcd(b+1,a)>1$, that $a\equiv 3\pmod 4$, and  that $b^n-1$ is a quadratic residue modulo  $a$ for every odd $n$. Infinitely many examples are obtained on taking $a$ a prime $\equiv 3\pmod 8$ and setting $b=a-1$.} and moreover such a proof will never yield for instance information on the magnitude of the denominator of $(a^n-1)/(b^n-1)$. 

\medskip

Here we shall obtain an effective results under the assumption that $r:=a/b$ is sufficiently close to $1$.   More precisely, we shall prove the following

\begin{theorem} Let $a,b$ be positive integers, such that $1<b<a<b(1+2^{-55})$. Then, for any integer $n>1$,  the denominator $d$ of ${a^n-1\over b^n-1}$ satisfies $d\ge \exp( \kappa n)$, for $\kappa=(2^{48}\log({1\over \log (a/b)}))^{-1}$.
\end{theorem}

Note that $\log (a/b)\ge \log ((b+1)/b)>1/2b$, hence $\kappa> (2^{49}\log b)^{-1}$. Since anyway $\kappa>0$ we have at once the following

\begin{corollary} Let $a,b$ be positive integers, such that $1<b<a<b(1+2^{-55})$. Then, for any integer $n>0$,  $b^n-1$ cannot divide $a^n-1$.
\end{corollary}

\begin{proof} [Proof of theorem]



The proof becomes a straightforward verification on appealing to  a lower bound for linear forms in logarithms of algebraic numbers, obtained by a refinement of  the method of {\sc A. Baker},  as stated in the book by {\sc Y. Bugeaud},  \cite{Bug}.
More precisely, we consider a special case of a theorem of {\sc M. Waldschmidt } appearing as Thm. 2.1, pp. 9--10 of the said book, actually the case where the coefficients denoted $\beta_i$ therein are rational, and where $n=2$, $\beta_0=0$, $\beta_2=1$. With these conventions, what  we need can be stated as follows, where we let $e=\exp(1)=2.7182...$:

\medskip

\noindent{\bf Lower Bound for a linear form}: {\it Let $\alpha_1,\alpha_2\ge 2$ be multiplicatively independent   rational numbers, and let $A_1,A_2 $  be positive real numbers with $\log A_j\ge \max(h(\alpha_j), e|\log\alpha_j|,1)$. Also let $c$ be a nonzero integer and suppose that  $C\ge \max(e, {1\over \log A_1}+ {|c|\over \log A_2})$. Then we have
$$
\log |c\log\alpha_1-\log\alpha_2|\ge -2^{43}\log A_1\log A_2\log C.
$$}

Coming back to our issue,  assume that 
$1<b<a<(1+\delta)b$, where $\delta=2^{-55}$ 
Then, recalling that $r:=a/b$, we have  $0<\log r<\log (1+\delta)<\delta$. Also, we must have $b\ge 2^{55}$, and we may also assume that $n\ge 2$.

Then, let us  put $q:={a^n-1\over b^n-1}$  and let $d$ denote its denominator.

\medskip

Note that  $1<r^n<q=r^n+{a^n-b^n\over b^n(b^n-1)}$. Also, since $a<2b\le b^2$, we have  ${a^n-b^n\over b^n-1}<{a^n\over b^n}=r^n$, so 
 $0<q-r^n<(r/b)^n$, 
Since  $\log q=\log (r^n+(q-r^n))=n\log r+\log (1+{q-r^n\over r^n})$, this implies in particular that $0<\log q-n\log r<b^{-n}$, whence
\begin{equation}\label{E.lb}
\log |n\log r-\log q|<-n\log b.
\end{equation}

To obtain a converse inequality we want to apply the stated Lower Bound, with appropriate data.

In the above statement we set $\alpha_1=r$, $\alpha_2=q$. Note that $\alpha_1=r=a/b>1$ is rational but not integral (since $1<r<2$), 
whereas $\alpha_2=q>1$. 
If $\alpha_1,\alpha_2$ would be multiplicatively dependent we would then  have $(a/b)^l=q^m$ for some positive   integers  $l,m>0$. Letting $p$ be a prime dividing the denominator of $r$, this prime would divide $b$, so would not divide $b^n-1$, which gives a contradiction. Therefore  the assumption of multiplicative independence holds.

\medskip

We have $h(\alpha_1)=\log a$ (since $a>b>1$) and  $h(\alpha_2)=\log dq$.    Further, we set $c=n$ in the statement. 

Let us  also set $\log A_1=2\log a$, $\log A_2=4\log dq$, $C=1/\log r$; then the assumptions are satisfied.  In fact,  $2\log a\ge \max (\log a,e\log r,1)$, since $r<2$, whereas $a>b\ge 2^{55}$, which proves the requirement for $A_1$. The one for $A_2$  reduces to  $\log A_2\ge \max(\log dq, e\log q, 1)$, which also holds.
As to  $C$ we have to check that ${1\over \log r}\ge \max (3, {1\over 2\log a}+{n\over 4\log dq})$. Now, ${1\over 2\log a}<1$, whereas $q>r^n$, so ${n\over 4\log dq}<{1\over 4\log r}$, and then  it suffices   that $\log r<1/3$, which holds. 

\medskip

The Lower Bound  then yields $\log |n\log r-\log q|\ge -2^{43}(2\log a)\cdot (4\log q)\cdot \log(1/\log r)$, whence, in view of \eqref{E.lb}, we obtain 
\begin{equation*}
n\log b\le  2^{46}\log a\cdot \log dq \cdot \log {1\over \log r}. 
\end{equation*} 

Now, $n\log r<\log q<\log (r^n(1+b^{-n}))\le n\log r+b^{-n}<2n\log r$, since $\log r\ge \log {b+1\over b}\ge b^{-2}\ge b^{-n}$.   Then, taking into account that $\log a\le 2\log b$ the last displayed  inequality implies
\begin{equation*}
1\le  
2^{48}{1\over f({1\over \log r})}+n^{-1}\cdot 2^{47}\log d\cdot  \log {1\over \log r}, 
\end{equation*} 
where we have denoted $f(x):= x/\log x$.

Now,  we have  $f(2^{55})=2^{55}/(55\log 2)>2^{49}$. Then since   $f(x)$ is increasing for $x>3$, and since $1/\log r>1/\delta=2^{55}$, we have $f(1/\log r)> 2^{49}$, hence $1\le   n^{-1}\cdot 2^{48}\log d\cdot  \log {1\over \log r}$, i.e.,
\begin{equation*}
\log d\ge {1\over 2^{48}  \log {1\over \log r}}\cdot n,
\end{equation*} 
which amounts to our conclusion. 
\end{proof}

 



\bigskip

\end{document}